\documentclass[oneside,english]{amsart}
\usepackage[T1]{fontenc}
\usepackage[latin9]{inputenc}
\usepackage{mathtools}
\usepackage{amsbsy}
\usepackage{amstext}
\usepackage{amsthm}
\usepackage{amssymb}
\usepackage{stackrel}
\usepackage{graphicx}
\usepackage{wasysym}

\makeatletter
\numberwithin{equation}{section}
\numberwithin{figure}{section}
\theoremstyle{plain}
\newtheorem{thm}{\protect\theoremname}
\theoremstyle{remark}
\newtheorem*{rem*}{\protect\remarkname}
\theoremstyle{plain}
\newtheorem{cor}[thm]{\protect\corollaryname}
\theoremstyle{plain}
\newtheorem{conjecture}[thm]{\protect\conjecturename}
\theoremstyle{plain}
\newtheorem*{fact*}{\protect\factname}
\theoremstyle{plain}
\newtheorem{lem}[thm]{\protect\lemmaname}

\usepackage{babel}

\providecommand{\conjecturename}{Conjecture}
\providecommand{\corollaryname}{Corollary}
\providecommand{\factname}{Fact}
\providecommand{\lemmaname}{Lemma}
\providecommand{\remarkname}{Remark}
\providecommand{\theoremname}{Theorem}

\makeatother

\usepackage{babel}
\providecommand{\conjecturename}{Conjecture}
\providecommand{\corollaryname}{Corollary}
\providecommand{\factname}{Fact}
\providecommand{\lemmaname}{Lemma}
\providecommand{\remarkname}{Remark}
\providecommand{\theoremname}{Theorem}

\begin{document}
\title[Extremal functions for the second-order Sobolev inequality]{Extremal functions for the second-order Sobolev inequality on groups
of polynomial growth}
\author{Bobo Hua, Ruowei Li and Florentin Münch}
\address{Bobo Hua: School of Mathematical Sciences, LMNS, Fudan University,
Shanghai 200433, China; Shanghai Center for Mathematical Sciences,
Jiangwan Campus, Fudan University, No. 2005 Songhu Road, Shanghai
200438, China.}
\email{bobohua@fudan.edu.cn}
\address{Ruowei Li: School of Mathematical Sciences, Fudan University, Shanghai
200433, China; Shanghai Center for Mathematical Sciences, Jiangwan
Campus, Fudan University, No. 2005 Songhu Road, Shanghai 200438, China.}
\email{rwli19@fudan.edu.cn}
\address{Florentin Münch: MPI MiS Leipzig, 04103 Leipzig, Germany}
\email{florentin.muench@mis.mpg.de}
\begin{abstract}
In this paper, we prove the second-order Sobolev inequalities on Cayley
graphs of groups of polynomial growth. We use the discrete Concentration-Compactness
principle to prove the existence of extremal functions for best constants
in supercritical cases. As applications, we get the existence of positive
ground state solutions to the $p$-biharmonic equations and the Lane-Emden
systems. 
\end{abstract}

\maketitle

\section{Introduction}

Sobolev inequalities are important in the studies of partial differential
equations and Riemannian geometry etc. For $N,\,k,\,p\geq1,\,kp<N,\frac{1}{\,p^{\star}}=\frac{1}{p}-\frac{k}{N}$,
we have the classical Sobolev inequality\label{eq:c-Sobo-ineq}, 
\begin{equation}
{\displaystyle \lVert u\rVert_{p^{\star}}\leq C_{p}\lVert u\rVert_{D^{k,p}},\;\forall u\in D_{0}^{k,p}\left(\mathbb{R}^{\mathit{N}}\right),}\label{eq:classical}
\end{equation}
where $C_{p}$ is a constant depending on $N$ and $p$, we omit the
dependence of constants $N$ for convenience, and $D_{0}^{k,p}\left(\mathbb{R}^{\mathit{N}}\right)$
denotes the completion of $C_{0}^{\infty}\left(\mathbb{R}^{\mathit{N}}\right)$
in the norm $\lVert u\rVert_{D^{k,p}}^{p}\coloneqq\sum\limits _{|\alpha|=k}\varint\lvert D^{\alpha}u\rvert^{p}\mathbf{\mathrm{d}}x$.
Define 
\[
D^{k,p}\left(\mathbb{R}^{\mathit{N}}\right)\coloneqq\left\{ u\in L^{\frac{Np}{N-kp}}(\mathbb{R}^{\mathit{N}}):\lVert u\rVert_{D^{k,p}}<\infty\right\} ,
\]
and it is well-known that $D_{0}^{k,p}\left(\mathbb{R}^{\mathit{N}}\right)=D^{k,p}\left(\mathbb{R}^{\mathit{N}}\right)$.
It is sufficient to establish the inequality (\ref{eq:classical})
for $k=1$ as its validity for higher $k$ can be obtained by induction
on $k$ \cite{RJ}.

Whether the best constant is attained by some $u\in D_{0}^{k,p}\left(\mathbb{R}^{\mathit{N}}\right)$,
which is called the extremal function, has been intensively investigated
in the literature. When $k=1,$ $p=1$ the best constant is the isoperimetric
constant and the extremal function is the characteristic function
of a ball \cite{FF,FR}. For $k=1,p>1$, the best constants and extremal
functions were obtained in \cite{A,T1,T2,R1,C1}. For $k>1,$ P. L.
Lions \cite{L1,L2,L3,L4} established the Concentration-Compactness
method, which provided a new idea for proving the existence of extremal
functions. The general idea is as follows. The best constant in the
Sobolev inequality (\ref{eq:classical}) is given by 
\[
K:=\inf_{\substack{u\in D^{k,p}(\mathbb{R}^{N})\\
\lVert u\rVert_{p^{\star}}=1
}
}\lVert u\rVert_{D^{k,p}}^{p}>0.
\]
Take a minimizing sequence $\{u_{n}\}$ and regard $\left\{ |u_{n}|^{p^{\star}}\mathrm{dx}\right\} $
as a sequence of probability measures. He proved in \cite{L1,L3}
that there are three cases of the limit of the sequence: compactness,
vanishing and dichotomy. Vanishing and dichotomy are ruled out by
the rescaling trick and subadditivity inequality. Therefore, the extremal
function exists by the compactness. Since the Concentration-Compactness
principle requires weak convergence $u_{n}\xrightharpoonup{}u\;\hphantom{}\text{in}\;D^{k,p}(\mathbb{R}^{\mathit{N}})$,
this method does not apply the case of $p=1$. And Lieb proved the
existence of extremal functions by a compactness technique \cite[Lemma 2.7]{L5}
which is induced by the Brézis-Lieb lemma \cite{BL}.

In recent years, people paid attention to the analysis on discrete
groups. Since the Sobolev inequalities are useful analytical tools,
they have been extended to the discrete setting \cite{CY,O1}. For
finite graphs, Sobolev inequalities and sharp constants have been
obtained by \cite{NKW,NKYTW,SST,TNK,YKWNT,YWK}. In this article,
we generalize our previous results on the existence of extremal functions
for the first-order Sobolev inequality \cite{HL} to higher order
inequalities on Cayley graphs of groups of polynomial growth.

Let $(G,S)$ be a Cayley graph of a group $G$ with a finite symmetric
generating set $S$, i.e. $S=S^{-1}.$ There is a natural metric on
$(G,S)$ called the word metric, denoted by $d^{S}$. Let $B_{p}^{S}(n):={\left\{ x\in G\mid d^{S}(p,x)\leq n\right\} }$
denote the closed ball of radius $n$ centered at $p\in G$ and denote
$\mid B_{p}^{S}(n)\mid:=\sharp B_{p}^{S}(n)$ as the volume (i.e.
cardinality) of the set $B_{p}^{S}(n)$. When $e$ is the unit element
of $G$, the volume $\beta_{S}(n):=\mid B_{e}^{S}(n)\mid$ of $B_{e}^{S}(n)$
is called the growth function of the group, see \cite{M68-1,M68-2,W68,G90,G97,G14,S55}.
A group $G$ is called of polynomial growth, or of polynomial volume
growth, if $\beta_{S}(n)\leq Cn^{A}$, for any $n\geq1$ and some
$A>0$, which is independent of the choice of the generating set $S$
since the metrics $d^{S}$ and $d^{S_{1}}$ are bi-Lipschitz equivalent
for different finite generating sets $S$ and $S_{1}$. By Gromov's
theorem and Bass' volume growth estimate of nilpotent groups \cite{B72},
for any group $G$ of polynomial growth there are constants $C_{1}(S)$,
$C_{2}(S)$ depending on $S$ and $N\in\mathbb{N}$ such that for
any $n\geq1$, 
\[
C_{1}(S)n^{N}\leq\beta_{S}(n)\leq C_{2}(S)n^{N},
\]
where the integer $N$ is called the homogeneous dimension or the
growth degree of $G$. Since $N$ is sort of dimensional constant
of $G$, we always omit the dependence of $N$ in various constants.

In this paper, we consider the Cayley graph $(G,S)$ of a group of
polynomial growth with the homogeneous dimension $N\geq3$. We denote
by $\ell^{p}(G)$ the $\ell^{p}$-summable functions on $G$ and by
$D_{0}^{k,p}(G)$ $(k=1,2)$ (resp. $\widetilde{D}_{0}^{2,p}(G)$
) the completion of finitely supported functions in the $D^{k,p}$
(resp. $\widetilde{D}^{2,p}$ ) norm, where $\lVert u\rVert_{D^{1,p}(G)}\coloneqq\lVert\nabla u\rVert_{\ell^{p}(G)}$,
$\lVert u\rVert_{D^{2,p}(G)}\coloneqq\lVert\Delta u\rVert_{\ell^{p}(G)}$
and $\lVert u\rVert_{\widetilde{D}^{2,p}(G)}\coloneqq\lVert\nabla^{2}u\rVert_{\ell^{p}(G)}$,
see Section 2 for details. Using Dungey's result \cite[Theorem 1]{N},
we prove the second-order norm defined by the Hessian and Laplace
operators are equivalent on a nilpotent group $G$, hence $D_{0}^{2,p}(G)=\widetilde{D}_{0}^{2,p}(G)$,
see Lemma \ref{lem: norm equivalence}. Analogous to the continuous
setting, set 
\[
D^{k,p}(G)\coloneqq\left\{ u\in\ell^{\frac{Np}{N-kp}}(G):\lVert u\rVert_{D^{k,p}(G)}<\infty\right\} ,
\]
\[
\widetilde{D}^{2,p}(G)\coloneqq\left\{ u\in\ell^{\frac{Np}{N-2p}}(G):\lVert u\rVert_{\widetilde{D}^{2,p}(G)}<\infty\right\} .
\]
For $\mathbb{R}^{N}$, the equivalence $D_{0}^{k,p}(\mathbb{R}^{N})=D^{k,p}(\mathbb{R}^{N})$
follows from the approximation of $f\in D^{k,p}(\mathbb{R}^{N})$
by compact supported functions using nice cutoff functions. Inspired
by the continuous setting, for a concrete example, integer lattice
graph $\mathbb{Z}^{N}$, the equivalence $D_{0}^{k,p}(\mathbb{Z}^{N})=D^{k,p}(\mathbb{Z}^{N})$
can be proved by constructing cutoff functions. We prove that any
$D^{1,p}(\mathbb{Z}^{N})$ function can be approximated by a \textquotedbl linear\textquotedbl{}
cutoff function of the logarithmic function (Theorem \ref{thm:the equivalence for Z^N_k=00003D00003D00003D1})
and Cosco, Nakajima and Schweiger give a more explicit estimate in
\cite{CSF21}. However, the \textquotedbl linear\textquotedbl{} cutoff
function does not work for $D^{2,p}(\mathbb{Z}^{N})$ since the Laplacian
is \textquotedbl bad\textquotedbl{} when it comes to the sharp corner.
Different from the continuous setting which can be mollified locally,
we find a way to fix this by taking a \textquotedbl close to linear\textquotedbl{}
third-order polynomial function being smooth at the sharp corner.
Moreover, the Euclidean distance is necessary for cutoff functions
since its Laplacian is decreasing while the combinatorial distance
is not. Then using the discrete chain rule (Lemma \ref{lem:chain rule})
we can get a good estimate of the second-order norm, and finally construct
desired cut-off functions, see Theorem \ref{thm: the equivalence for Z^N_k=00003D00003D00003D2}.
Then we prove the equivalence $D_{0}^{k,p}(G)=D^{k,p}(G)$ on general
Cayley graphs $(G,S)$, see Theorem \ref{thm: the equivalence for G}
and Theorem \ref{thm: the equivalence for G_k=00003D00003D00003D2}.
First $D_{0}^{k,p}(G)\subseteq D^{k,p}(G)$ follows from Sobolev inequalities,
then we prove the converse direction via the parabolicity theory on
graphs \cite[Corollary 2.6]{SL} for $k=1$ and the functional analysis
method by checking the Laplace operator is an isometry from $D_{0}^{2,p}(G)$
(also $D^{2,p}(G)$) to $\ell^{p}(G)$.

For special case $p=2$, we extend the results to the Cayley graph
$(G,S)$ satisfying $\beta_{S}(n)\geq C(S)n^{N},\ \forall n\geq1$.
We prove the Hodge decomposition theorem on $1$-forms on edges (Theorem
\ref{thm:Hodge decomposition}) and use it to prove $D_{0}^{1,2}=D^{1,2}$
on $G$ (Corollary \ref{cor:D012=00003D00003D00003D00003DD12}). Moreover,
we prove $D_{0}^{2,2}=D^{2,2}$ on the groups satisfying the second-order
Sobolev inequality including polynomial growth groups and non-amenable
groups.

Using the boundedness of Riesz transforms \cite{N} and the functional
calculus in Banach spaces \cite{M} we get the following discrete
second-order Sobolev inequality by induction on the first-order Sobolev
inequality: 
\begin{equation}
\lVert u\rVert_{\ell^{q}}\leq C_{p,q}\lVert u\rVert_{D^{2,p}},\;\forall u\in D^{2,p}(G),\label{eq:second-order Sobo ineq}
\end{equation}
where $N\geq3,1<p<\dfrac{N}{2},q=p^{\ast\ast}\coloneqq\dfrac{Np}{N-2p}$,
see Theorem \ref{thm: second order sobo ineq}. Since $\ell^{p}(G)$
embeds into $\ell^{q}(G)$ for any $q>p$, one verifies that the inequality
(\ref{eq:second-order Sobo ineq}) holds for $q\geq p^{\ast\ast}$.
Recalling the continuous setting, it is called subcritical for $q<p^{\ast\ast}$,
critical for $q=p^{\ast\ast}$ and supercritical for $q>p^{\ast\ast}$
for Sobolev inequalities. Therefore, (\ref{eq:second-order Sobo ineq})
holds in both critical and supercritical cases.

The optimal constant in the Sobolev inequality (\ref{eq:second-order Sobo ineq})
is given by 
\begin{equation}
K:=\inf_{\substack{u\in D^{2,p}(G)\\
\lVert u\rVert_{q}=1
}
}\lVert u\rVert_{D^{2,p}}^{p}.\label{eq:inf}
\end{equation}
In order to prove that the infimum is achieved, we consider a minimizing
sequence $\{u_{n}\}\subset D^{2,p}(G)$ satisfying 
\begin{equation}
\lVert u_{n}\rVert_{q}=1,\lVert u_{n}\rVert_{D^{2,p}}^{p}\to K,n\to\infty.\label{eq:min seq}
\end{equation}
We want to prove $u_{n}\to u$ strongly in $D^{2,p}(G)$, which yields
that $u$ is a minimizer.

We prove the following main results. 
\begin{thm}
\label{thm:main1}For $N\geq3,1<p<\frac{N}{2}$, $q>p^{\ast\ast}=\frac{Np}{N-2p}$,
let $\left\{ u_{n}\right\} \subset D^{2,p}(G)$ be a minimizing sequence
satisfying (\ref{eq:min seq}). Then there exists a sequence $\{x_{n}\}\subset G$
and $v\in D^{2,p}(G)$ such that the sequence after translation $\left\{ v_{n}(x):=u_{n}(x_{n}x)\right\} $
contains a convergent subsequence that converges to $v$ in $D^{2,p}(G)$.
And $v$ is a minimizer for $K$. 
\end{thm}

\begin{rem*}
(1) This result implies that best constant can be obtained in the
supercritical case.

(2) If we define the second-order norm by the Hessian operator, i.e.
$\tilde{D}^{2,p}$ (see Section 2), then the second-order Sobolev
inequalities also hold, and the strong convergence and existence of
minimizer in $\tilde{D}_{0}^{2,p}(G)$ follows from a similar argument.

(3) In particular, for $p=1$, we can get the same results in $\tilde{D}^{2,1}(\mathbb{Z}^{N})$,
see Section 7. 
\end{rem*}
We will provide two proofs for Theorem \ref{thm:main1}. In the continuous
setting, Lions proved the existence of extremal functions by Concentration-Compactness
principle \cite[Lemma I.1.]{L3} and a rescaling trick \cite[Theorem I.1, (17)]{L3}.
And Lieb in \cite{L5} used a compactness technique and the rearrangement
inequalities. Following Lions, the main idea of proof I is to prove
a discrete analog of Concentration-Compactness principle, see Lemma
\ref{lem:Concentration-Compact}. However, we don't know proper notion
of rescaling and rearrangement tricks on graphs to exclude the vanishing
case of the limit function. Inspired by \cite{HLY}, for the supercritical
case, we prove that the translation sequence has a uniform positive
lower bound at the unit element, see Lemma \ref{lem:lower bound},
which excludes the vanishing case. The idea of proof II is based on
a compactness technique by Lieb \cite[Lemma 2.7]{L5} and the nonvanishing
of the limit of translation sequence.

As applications, we get the existence of positive ground state solutions
for the discrete nonlinear $p$-biharmonic equations and the Lane--Emden
systems in supercritical cases. They have certain physical backgrounds,
such as traveling waves in a suspension bridge \cite{Z} and the static
deflection of an elastic plate \cite{IA}, and have been well studied
in continuous setting, see for example \cite{L6,CL,WX,SZ1,SZ2,SZ3,BM,PQS,M1,M2,S1,K1,MC,GG,dF,W10,WY13,HHY14,LY16,MY19,MHY19,NNPY20,GHZ15,HY15,HY16,HYZ18}
and references therein. For discrete setting, there are few results
of the fourth order nonlinear equations \cite{HSZ,XZ21}.

First by Theorem \ref{thm:main1}, we can get the existence of positive
solutions to Euler-Lagrange equation of (\ref{eq:second-order Sobo ineq})
as follows. 
\begin{cor}
\label{cor:p-bihar}For $N\geq3,1<p<\frac{N}{2},q>p^{\ast\ast}$,
there is a positive ground state solution of the nonlinear $p$-biharmonic
equation 
\begin{equation}
\Delta\left(\mid\Delta u\mid^{p-2}\Delta u\right)-\mid u\mid^{q-2}u=0,\;\forall u\in D^{2,p}(G),\label{eq:p-bihar}
\end{equation}
where $\Delta$ is the Laplace on graphs defined as $\Delta u(x)\coloneqq\sum\limits _{y\sim x}(u(y)-u(x))$. 
\end{cor}

Consider the Lane-Emden system 
\begin{equation}
\begin{cases}
\begin{array}{cc}
-\Delta u=\lvert v\rvert^{p^{\prime}-2}v\\
-\Delta v=\lvert u\rvert^{q-2}u\\
u\in D^{2,p}(G),v\in D^{2,q^{\prime}}(G),
\end{array}\end{cases}\label{eq:LE system}
\end{equation}
where $N\geq3$, $p,q>1$, $p^{\prime}=\frac{p}{p-1}$, $q^{\prime}=\frac{q}{q-1}$,
$(p^{\prime},q)$ lies on the critical hyperbola, that is, 
\[
\frac{1}{p^{\prime}}+\frac{1}{q}=\frac{N-2}{N}.
\]
Indeed, $(u,v)$ is a solution of (\ref{eq:LE system}) if and only
if $u$ is a solution of (\ref{eq:p-bihar}) with $v\coloneqq-\mid\Delta u\mid^{p-2}\Delta u$
inB critical cases. This relation is sometimes called reduction-by-inversion,
see \cite[Lemma 2.1]{CS} \cite{BdT,dD,R2} for continuous setting.
We can also prove the equivalence in supercritical cases for discrete
setting, see Lemma \ref{lem: the equivalence}. Hence, by Corollary
\ref{cor:p-bihar} we have a positive solution of the Lane-Emden system. 
\begin{cor}
\label{cor:L-E system}For $N\geq3$, $p,q>1$, $\frac{1}{p^{\prime}}+\frac{1}{q}<\frac{N-2}{N}$,
there is a pair of positive solution $(u,v)$ for the Lane-Emden system
(\ref{eq:LE system}). 
\end{cor}

\begin{conjecture}
According to the results in continuous cases \cite[Corollary I.2]{L3}
\cite{L6,CL,WX,SZ1,BM,PQS,M1,S1,dF,MY19,WY13,W10}, we conjecture
that (\ref{eq:p-bihar}) and (\ref{eq:LE system}) have positive solutions
in critical cases and the non-negative solutions in subcritical cases
are trivial. 
\end{conjecture}

The paper is organized as follows. In Section 2, we recall some basic
facts and prove some useful lemmas. In Section 3, we study the important
equivalence of Sobolev spaces and prove the discrete second-order
Sobolev inequality using the boundedness of Riesz transforms and the
functional calculus. In Section 4, we introduce the Brézis-Lieb lemma
and prove the Concentration-Compactness principle on $G$. In Section
5, we prove a key lemma to exclude the vanishing case and give two
proofs for Theorem \ref{thm:main1}. As applications, we get the existence
results for $p$-biharmonic equations and Lane-Emden systems in Section
6. In Section 7, we prove the Hodge decomposition theorem on $1$-forms
on edges for $p=2$ and study the existence of extremal functions
for $p=1$.

\section{Preliminaries}

Let $G$ be a countable group. It is called finitely generated if
it has a finite generating set $S$. We always assume that the generating
set $S$ is symmetric, i.e. $S=S^{-1}$. The Cayley graph of $(G,S)$
is a graph structure $(V,E)$ with the set of vertices $V=G$ and
the set of edges $E$ where for any $x,y\text{\ensuremath{\in}}G,xy\in E$
(also denoted by $x\sim y$) if $x=ys$ for some $s\in S$. The Cayley
graph of $(G,S)$ is endowed with a natural metric, called the word
metric \cite{BBI01}: For any $x,y\in G$, the distance between them
is defined as the length of the shortest path connecting $x$ and
$y$ by assigning each edge of length one, 
\[
d^{S}(x,y)=\text{inf}{\left\{ k:x=x_{0}\sim\cdot\cdot\cdot\sim x_{k}=y\right\} .}
\]
One easily sees that for two generating sets $S$ and $S_{1}$ the
metrics $d^{S}$ and $d^{S_{1}}$ are bi-Lipschitz equivalent, i.e.
there exist two constants $C_{1}(S,S_{1}),C_{2}(S,S_{1})$ such that
for any $x,y\in G$ 
\[
C_{1}(S,S_{1})d^{S_{1}}(x,y)\leq d^{S}(x,y)\leq C_{2}(S,S_{1})d^{S_{1}}(x,y).
\]
Let $B_{p}^{S}(n):={\left\{ x\in G\mid d^{S}(p,x)\leq n\right\} }$
denote the closed ball of radius $n$ centered at $p\in G$. By the
group structure, it is obvious that $\mid B_{p}^{S}(n)\mid=\mid B_{q}^{S}(n)\mid$,
for any $p,q\in G$. The growth function of $(G,S)$ is defined as
$\beta_{S}(n):=\mid B_{e}^{S}(n)\mid$ where $e$ is the unit element
of $G$. A group $G$ is called of polynomial growth if there exists
a finite generating set $S$ such that $\beta_{S}(n)\leq Cn^{A}$
for some $C,A>0$ and any $n\geq1$. One checks that this definition
is independent of the choice of the generating set $S$. Thus, the
polynomial growth is indeed a property of the group $G$. In this
paper, we consider the Cayley graph $(G,S)$ of a group of polynomial
growth 
\[
C_{1}(S)n^{N}\leq\beta_{S}(n)\leq C_{2}(S)n^{N},
\]
for some $N\in\mathbb{N}$ and any $n\geq1$, where $N$ is called
the homogenous dimension of $G$.

We denote by $C(G)$ the space of functions on $G$. The support of
$u\in C(G)$ is defined as $\textrm{supp}(u)\coloneqq\{x\in G:u(x)\neq0\}$.
Let $C_{0}(G)$ be the set of all functions with finite support. For
any $u\in C(G)$, the $\ell^{p}$ norm of $u$ is defined as 
\[
\lVert u\rVert_{\ell^{p}(G)}\coloneqq\begin{cases}
\left(\underset{x\in G}{\sum}\lvert u(x)\rvert^{p}\right)^{1/p} & \text{\ensuremath{0<p<\infty,}}\\
\underset{x\in G}{\sup}\lvert u(x)\rvert & p=\infty,
\end{cases}
\]
and we shall write $\lVert u\rVert_{{\ell^{p}(G)}}$ as $\lVert u\rVert_{{p}}$
for convenience. The $\ell^{p}(G)$ space is defined as 
\[
\ell^{p}(G)\coloneqq\left\{ u\in C(G):\lVert u\rVert_{{\ell^{p}(G)}}<\infty\right\} .
\]
For any $u\in C(G)$, the difference operator is defined as for any
$x\sim y$

\[
\nabla_{xy}u=u(y)-u(x).
\]
In particular, for the Cayley graph $(G,S)$ with symmetric generating
set $S\coloneqq\left\{ s_{1},s_{2},\cdots,s_{m}\right\} $, we can
define the difference in $i$-th direction as 
\[
\partial_{i}u(x)\coloneqq u(xs_{i})-u(x).
\]
Let 
\[
\lvert\nabla u(x)\rvert_{p}\coloneqq\left(\sum\limits _{y\sim x}\lvert\nabla_{xy}u\rvert^{p}\right)^{1/p}
\]
be the $p$-norm of the gradient of $u$ at $x$.

We define the Laplace operator as 
\[
\Delta u(x)\coloneqq\sum\limits _{y\sim x}(u(y)-u(x)).
\]
We write $p$-norm of the Hessian of $u$ as 
\[
\lVert\nabla^{2}u\rVert_{p}\coloneqq\left(\stackrel[i,j=1]{m}{\sum}\underset{x\in G}{\sum}\lvert\partial_{j}\partial_{i}u(x)\lvert^{p}\right)^{1/p}=\left(\stackrel[i,j=1]{m}{\sum}\lVert\partial_{j}\partial_{i}u(x)\lVert_{p}^{p}\right)^{1/p}.
\]
The $D^{k,p}$ ($k=1,2$) norms of $u$ are given by 
\[
\lVert u\rVert_{D^{1,p}(G)}\coloneqq\lVert\lvert\nabla u\rvert_{p}^{p}\rVert_{\ell^{1}(G)}^{1/p}=\left(\sum\limits _{x\in G}\sum\limits _{y\sim x}\lvert\nabla_{xy}u\rvert^{p}\right)^{1/p},
\]
\[
\lVert u\rVert_{D^{2,p}(G)}\coloneqq\lVert\Delta u\rVert_{\ell^{p}(G)}=\left(\sum\limits _{x\in G}\lvert\sum\limits _{y\sim x}\nabla_{xy}u\rvert^{p}\right)^{1/p}.
\]
We use the Hessian operator to define $\widetilde{D}^{2,p}$ norm
as 
\[
\lVert u\rVert_{\widetilde{D}^{2,p}(G)}\coloneqq\lVert\nabla^{2}u\rVert_{p}.
\]

Since $G$ is a regular graph, that is, the degree $d_{x}\coloneqq\sharp\left\{ y:y\sim x\right\} $
is constant for $x\in G$, we have 
\[
\lVert\nabla u\lVert_{p}\coloneqq\lVert\lvert\nabla u(x)\rvert_{1}\lVert_{p}\sim\lVert\lvert\nabla u(x)\rvert_{p}^{p}\lVert_{1}^{1/p}=\lVert u\rVert_{D^{1,p}}.
\]
In this paper, we use $a\lesssim b$ to denote $a\leq Cb$ for some
$C>0$ and $a\sim b$ to denote $a\lesssim b\lesssim a$.

We define $D_{0}^{k,p}(G)$ ($k=1,2$) and $\widetilde{D}_{0}^{2,p}(G)$
as the completion of $C_{0}\left(G\right)$ in $D^{k,p}$ norm and
$\widetilde{D}^{2,p}$ norm respectively. And define 
\[
D^{k,p}(G)\coloneqq\left\{ u\in\ell^{\frac{Np}{N-kp}}(G):\lVert u\rVert_{D^{k,p}(G)}<\infty\right\} ,
\]
\[
\widetilde{D}^{2,p}(G)\coloneqq\left\{ u\in\ell^{\frac{Np}{N-2p}}(G):\lVert u\rVert_{\widetilde{D}^{2,p}(G)}<\infty\right\} .
\]
Let $c_{0}(G)$ be the completion of $C_{0}\left(G\right)$ in $\ell^{\infty}$
norm. It is well-known that $\ell^{1}(G)=(c_{0}(G))^{\ast}$. We set
\[
\lVert\mu\rVert\coloneqq\sup\limits _{u\in c_{0}(G),\lVert u\rVert_{\infty}=1}\langle\mu,u\rangle,\qquad\forall\mu\in\ell^{1}(G).
\]
By definition, 
\[
\mu_{n}\xrightharpoonup{w^{\ast}}\mu\;\text{in}\;\ell^{1}(G)\;\text{if and only if}\;\langle\mu_{n},u\rangle\longrightarrow\langle\mu,u\rangle,\forall u\in c_{0}(G).
\]
In the proof, we will use the following results, see \cite{C2}. 
\begin{fact*}
(a)Every bounded sequence of $\ell^{1}(G)$ contains a $w^{\ast}$-convergent
subsequence. \\
 (b)If $\mu_{n}\xrightharpoonup{w^{\ast}}\mu$ in $\ell^{1}(G)$,
then ${\mu_{n}}$ is bounded and 
\[
\lVert\mu\rVert\leq\varliminf\limits _{n\to\infty}\lVert\mu_{n}\rVert.
\]
(c)If $\mu\in\ell^{1+}(G):=\left\{ \mu\in\ell^{1}(G):\mu\geq0\right\} $,
then 
\[
\lVert\mu\rVert=\langle\mu,1\rangle.
\]
\end{fact*}
Next we give a discrete chain rule. 
\begin{lem}
\label{lem:chain rule}For a weighted graph $G=(V,w,m)$, where $w_{xy}$,
$m(x)$ denote the weights assigned to the edge $xy\in E$ and to
the vertex $v$, respectively. Let $\phi\in C^{2}(\mathbb{R})$, $f\in C(V)$,
then for some\textup{ $\xi\in[m,M]$,} 
\[
\Delta\phi(f)(x)=\phi'(f(x))\Delta f(x)+\phi''(\xi)\Gamma f(x),
\]
where $\Delta f(x)\coloneqq\underset{y\sim x}{\sum}\frac{w_{xy}}{m(x)}\left(f\left(y\right)-f(x)\right)$,
$\Gamma f(x)\coloneqq\frac{1}{2}\underset{y\sim x}{\sum}\frac{w_{xy}}{m(x)}\mid f(y)-f(x)\mid^{2}$
and $m\coloneqq\underset{y\in B_{x}(1)}{\textrm{min}}f(y)$, $M\coloneqq\underset{y\in B_{x}(1)}{\text{max}}f(y)$. 
\end{lem}

\begin{proof}
With $s=f(x)$ and $t_{y}=f(y)-f(x)$, Taylor's theorem gives 
\[
\phi(s+t_{y})=\phi(s)+t_{y}\phi'(s)+\frac{1}{2}t_{y}^{2}\phi''(\xi_{y})
\]
for some $\xi_{y}\in[s\wedge(s+t_{y}),s\vee(s+t_{y})]$. Summing up
gives 
\begin{align*}
\Delta\phi(f)(x) & =\underset{y\sim x}{\sum}\frac{w_{xy}}{m(x)}\left(\phi\left(s+t_{y}\right)-\phi(s)\right)\\
 & =\underset{y\sim x}{\sum}\frac{w_{xy}}{m(x)}\left(t_{y}\phi'\left(s\right)+\frac{1}{2}t_{y}^{2}\phi''(\xi_{y})\right)\\
 & =\underset{y\sim x}{\sum}\frac{w_{xy}}{m(x)}\left(t_{y}\phi'\left(s\right)+\frac{1}{2}t_{y}^{2}\phi''(\xi)\right)\\
 & =\phi'(f(x))\Delta f(x)+\phi''(\xi)\Gamma f(x),
\end{align*}
where in the third identity, we can choose some $\xi$ between minimum
and maximum of the $\xi_{y}$ by continuity of $\phi''$, that is,
\[
m\leq\xi\leq M.
\]
This finishes the proof. 
\end{proof}

\section{Sobolev spaces and second-order Sobolev inequalities}

In this section, inspired by the results in the continuous setting,
we study some important properties of Sobolev spaces defined in Section
2. By induction on the discrete first-order Sobolev inequalities,
we prove the second-order Sobolev inequalities using the boundedness
of Riesz transforms and the functional calculus.

First, we introduce a key lemma of norm equivalence using Dungey's
result \cite[Theorem 1]{N}. 
\begin{lem}
\label{lem: norm equivalence}Let Cayley graph $(G,S)$ is a nilpotent
group of polynomial growth with symmetric generating set $S\coloneqq\left\{ s_{1},s_{2},\cdots,s_{m}\right\} $,
for $p>1,$ if $u\in C_{0}(G)$, then 
\[
\lVert u\rVert_{\widetilde{D}^{2,p}(G)}\sim\lVert u\rVert_{D^{2,p}(G)}.
\]
Moreover, $D_{0}^{2,p}(G)=$$\widetilde{D}_{0}^{2,p}(G)$. 
\end{lem}

\begin{proof}
Since $-\Delta u(x)=\frac{1}{2}\stackrel[i=1]{m}{\sum}\partial_{s_{i}^{-1}}\partial_{s_{i}}u(x),$
then 
\begin{align*}
\lVert\Delta u\rVert_{p} & =\left(\underset{x}{\sum}\mid\frac{1}{2}\stackrel[i=1]{m}{\sum}\partial_{s_{i}^{-1}}\partial_{s_{i}}u(x)\mid^{p}\right)^{1/p}\\
 & \lesssim\left(\underset{x}{\sum}\stackrel[i=1]{m}{\sum}\mid\partial_{s_{i}^{-1}}\partial_{s_{i}}u(x)\mid^{p}\right)^{1/p}\\
 & \leq\left(\underset{x}{\sum}\stackrel[i,j=1]{m}{\sum}\mid\partial_{s_{j}}\partial_{s_{i}}u(x)\mid^{p}\right)^{1/p}=\lVert\nabla^{2}u\rVert_{p}.
\end{align*}
Using Dungey's result in \cite[Theorem 1]{N}, we know for any $i,j$,
\begin{equation}
\lVert\partial_{s_{j}}\partial_{s_{i}}u(x)\rVert_{p}\lesssim\lVert\Delta u\rVert_{p},\label{eq: hessian_laplace inequality}
\end{equation}
which implies 
\[
\lVert\nabla^{2}u\rVert_{p}\lesssim\lVert\Delta u\rVert_{p}.
\]
Hence the norms defined by the Laplace and Hessian operator are equivalent.
So that the spaces defined by the completion of finitely supported
functions in the equivalent norms coincide. 
\end{proof}
\begin{rem*}
The assumption of nilpotent group is necessary, if $G$ is not nilpotent
then, (\ref{eq: hessian_laplace inequality}) may fail, see \cite[Section 1]{A02}. 
\end{rem*}
For a Cayley graph it is well known that if $\beta_{S}(n)\geq C(S)n^{N},\ \forall n\geq1$
for $N\geq3,$ then the first-order Sobolev inequality holds \cite[Theorem 3.6]{HM},
\begin{equation}
{\displaystyle \lVert u\rVert_{\ell^{q}}\leq C_{p,q}\lVert u\rVert_{D^{1,p}},\;\forall u\in D_{0}^{1,p}(G)}\label{eq:first-order sobo ineq-1}
\end{equation}
where $1\leq p<N,q\geq p^{\ast}\coloneqq\dfrac{Np}{N-p}.$ In fact,
this follows from a standard trick and the isoperimetric estimate
\cite[Theorem 4.18]{W1}.

We consider a concrete example of Cayley graphs of polynomial growth
with the homogeneous dimension $N$, the integer lattice graph $\mathbb{Z}^{N}$,
which serve as the discrete counterpart of $\mathbb{R}^{N}$. It consists
of the set of vertices $V=\mathbb{Z}^{N}$ and the set of edges 
\[
E=\left\{ \left\{ x,y\right\} :x,y\in\mathbb{Z}^{N},\mathop{\sum\limits _{i=1}^{N}\lvert x_{i}-y_{i}\rvert=1}\right\} .
\]
Analogous the continuous setting, we construct cutoff functions to
prove the results on $\mathbb{Z}^{N}$. 
\begin{thm}
\label{thm:the equivalence for Z^N_k=00003D00003D00003D1}If $N\geq3$,
$1<p<N$, then $D_{0}^{1,p}(\mathbb{Z}^{N})=D^{1,p}(\mathbb{Z}^{N}).$ 
\end{thm}

\begin{proof}
For any $u\in D_{0}^{1,p}(\mathbb{Z}^{N})$, there exists a sequence
$\left\{ u_{n}\right\} \subset C_{0}\left(\mathbb{Z}^{N}\right)$
such that 
\[
u_{n}\stackrel{}{\longrightarrow}u\text{ in }D^{1,p}.
\]
And $u\in\ell^{\frac{Np}{N-p}}(\mathbb{Z}^{N})$ by the Sobolev inequality
(\ref{eq:first-order sobo ineq-1}). Hence, $D_{0}^{1,p}(\mathbb{Z}^{N})\subseteq D^{1,p}(\mathbb{Z}^{N})$.

In the other direction, the key is to find suitable cutoff functions
$\eta_{n}(x)\in C_{0}\left(\mathbb{Z}^{N}\right)$. For any $u\in D^{1,p}(\mathbb{Z}^{N})$,
set $u_{n}\coloneqq u\eta_{n}\in C_{0}\left(\mathbb{Z}^{N}\right)$,
then by H$\ddot{\text{o}}$lder inequalities 
\begin{align}
\lVert u_{n}-u\rVert_{D^{1,p}(\mathbb{Z}^{N})}^{p} & =\sum\limits _{x\in\mathbb{Z}^{N}}\sum\limits _{y\sim x}\lvert\nabla_{xy}(u\eta_{n})-\nabla_{xy}u\rvert^{p}\nonumber \\
 & =\sum\limits _{x\in\mathbb{Z}^{N}}\sum\limits _{y\sim x}\lvert\nabla_{xy}u\eta_{n}(y)+\nabla_{xy}\eta_{n}u(x)-\nabla_{xy}u\rvert^{p}\label{eq:k=00003D00003D00003D1-1}\\
 & \lesssim\sum\limits _{x\in\mathbb{Z}^{N}}\lvert\nabla u(x)\mid_{p}^{p}\underset{y\sim x}{\text{max}}\mid\eta_{n}(y)-1\mid^{p}+\lVert\nabla\eta_{n}\rVert_{\ell^{N}}^{p}\lVert u\rVert_{\ell^{p^{\ast}}}^{p},\nonumber 
\end{align}
Hence, if the cutoff functions satisfy 
\[
\eta_{n}\text{ is uniformly bounded,}
\]
\begin{equation}
\eta_{n}\to1\text{ pointwise on \ensuremath{\mathbb{Z}^{N}}}\label{eq: cut fun req 1-1}
\end{equation}
and 
\[
\lVert\nabla\eta_{n}\rVert_{\ell^{N}(\mathbb{Z}^{N})}\rightarrow0,
\]
then we can prove the other direction by the dominated convergence
theorem.

Let $r>1$, and $R\gg r$ be large enough. Define 
\[
\eta(x)\coloneqq1\wedge\frac{\textrm{log\ensuremath{R}}-\text{log\ensuremath{\mid x\mid}}}{\textrm{log\ensuremath{R}}-\text{log\ensuremath{r}}}\vee0,
\]
where $\mid x\mid$ stands for the Euclidean distance. Then

\begin{align*}
\lVert\nabla\eta\rVert_{\ell^{N}(\mathbb{Z}^{N})}^{N} & =\sum\limits _{r\leq\mid x\mid\leq R}\sum\limits _{y\sim x}\lvert\nabla_{xy}\eta\rvert^{N}\\
 & \lesssim\left(\textrm{log\ensuremath{\frac{R}{r}}}\right)^{-N}\sum\limits _{r-1\leq\mid x\mid\leq R+1}\sum\limits _{y\sim x}\lvert\textrm{log\ensuremath{\mid x\mid}}-\text{log\ensuremath{\ensuremath{\mid y\mid}}}\rvert^{N}\\
 & \lesssim\left(\textrm{log\ensuremath{\frac{R}{r}}}\right)^{-N}\sum\limits _{r-1\leq\mid x\mid\leq R+1}\lvert x\rvert^{-N},
\end{align*}
where the second inequality follows from the mean value theorem. In
the following, we can estimate the summation on $\mathbb{Z}^{N}$
by the integral on $\mathbb{R}^{\mathit{N}}$, 
\begin{align*}
\sum\limits _{r-1\leq\mid x\mid\leq R+1}\lvert x\rvert^{-N} & \lesssim\sum\limits _{r-1\leq\mid x\mid\leq R+1}\int\limits _{S_{x}(\frac{1}{2})}\lvert x\rvert^{-N}\text{d\ensuremath{t}}\\
 & \lesssim\sum\limits _{r-1\leq\mid x\mid\leq R+1}\int\limits _{S_{x}(\frac{1}{2})}\lvert t\rvert^{-N}\text{d\ensuremath{t}}\\
 & \lesssim\int\limits _{\widetilde{B}(R+2)\setminus\widetilde{B}(r-2)}\lvert t\rvert^{-N}\text{d\ensuremath{t}},
\end{align*}
where $S_{x}(\frac{1}{2})\coloneqq\left\{ t\in\mathbb{R}^{\mathit{N}}:\lvert t_{i}-x_{i}\rvert<\frac{1}{2},1\leq i\leq N\right\} $
is the Euclidean cube, and $\widetilde{B}(r)$ is the Euclidean ball
in $\mathbb{R}^{\mathit{N}}$ with radius of $r$ and centered at
the origin. Hence, 
\[
\lVert\nabla\eta\rVert_{\ell^{N}(\mathbb{Z}^{N})}^{N}\lesssim\left(\textrm{log\ensuremath{\frac{R}{r}}}\right)^{-N}\text{log}\ensuremath{\frac{R}{r}}=O\left(\textrm{\ensuremath{\left(\text{log}\ensuremath{\frac{R}{r}}\right)}}^{1-N}\right).
\]
For fixed $r$, letting $R\to\infty$ when $N\geq3$, we have that
$\lVert\nabla\eta\rVert_{N}\to0$. Since $\eta(x)$ is uniformly bounded
and tends to $1$ pointwise, we prove $u\in D_{0}^{1,p}(\mathbb{Z}^{N})$
by (\ref{eq:k=00003D00003D00003D1-1}). 
\end{proof}
\begin{rem*}
For $D^{1,p}(\mathbb{Z}^{N})$, one can replace the Euclidean distance
$\mid\cdot\mid$ in the definition of cutoff function $\eta$ by any
norm $\parallel\cdot\parallel_{r}$ with $r>0$, and Cosco, Nakajima
and Schweiger give a more explicit estimate of the $p$-capacity \cite{CSF21}. 
\end{rem*}
For a Cayley graph $(G,S)$ of polynomial growth with the homogeneous
dimension $N$, it is difficult to construct desired cutoff functions.
We find an alternative method to prove the result by the parabolicity
theory \cite{SL}. 
\begin{thm}
\label{thm: the equivalence for G}If $N\geq3$, $1<p<N$, then $D_{0}^{1,p}(G)=D^{1,p}(G)$. 
\end{thm}

\begin{proof}
By the same argument as Theorem \ref{thm:the equivalence for Z^N_k=00003D00003D00003D1},
$D_{0}^{1,p}(G)\subseteq D^{1,p}(G)$ follows from the Sobolev inequalities
(\ref{eq:first-order sobo ineq-1}).

In the other direction, by the parabolicity theory on graphs \cite[Corollary 2.6]{SL}
we know that $G$ is $N$-parabolic which implies the $N$-capacity
is zero. Hence, we can get a uniformly bounded sequence $\eta_{n}(x)$
with finite support satisfying $\eta_{n}(x)\to1$ pointwise and $\lVert\nabla\eta_{n}\rVert_{\ell^{N}(G)}\rightarrow0$.
Set $u_{n}\coloneqq u\eta_{n}\in C_{0}(G)$, then by H$\ddot{\text{o}}$lder
inequalities 
\begin{align*}
\lVert u_{n}-u\rVert_{D^{1,p}(G)}^{p} & =\sum\limits _{x\in G}\sum\limits _{y\sim x}\lvert\nabla_{xy}u\eta_{n}(y)+\nabla_{xy}\eta_{n}u(x)-\nabla_{xy}u\rvert^{p}\\
 & \lesssim\sum\limits _{x\in G}\lvert\nabla u(x)\mid_{p}^{p}\underset{y\sim x}{\text{max}}\mid\eta_{n}(y)-1\mid^{p}+\lVert\nabla\eta_{n}\rVert_{\ell^{N}}^{p}\lVert u\rVert_{\ell^{p^{\ast}}}^{p}\to0.
\end{align*}
That is, $D^{1,p}(G)\subseteq D_{0}^{1,p}(G)$. 
\end{proof}
\begin{rem*}
For $N\geq3,$ if the Cayley graph $(G,S)$ satisfies $\beta_{S}(n)\geq C(S)n^{N},\ \forall n\geq1$,
then we can prove $D_{0}^{1,2}(G)=D^{1,2}(G)$ using the Hodge decomposition
theorem on $1$-forms on edges, see Theorem \ref{thm:Hodge decomposition}
and Corollary \ref{cor:D012=00003D00003D00003D00003DD12}. 
\end{rem*}
By the properties of semigroup \cite[Chapter 5]{B}, we have the following
lemma. 
\begin{lem}
\label{lem:L^(1/2) is bounded}For any $u\in\ell^{p}(G)$, $p\in(1,\infty)$,
\[
Mu\coloneqq\sideset{\frac{1}{\mid\varGamma(-\frac{1}{2})\mid}}{_{0}^{\infty}}\int\left(e^{t\Delta}u-u\right)t^{-\frac{3}{2}}\text{d}t,
\]
where $e^{t\Delta}$ is the semigroup of $\Delta$. Then $M$ is a
bounded linear operator in $\ell^{p}(G)$, and 
\[
L^{\frac{1}{2}}u\coloneqq(-\Delta)^{\frac{1}{2}}u=Mu,\forall u\in\ell^{p}(G).
\]
\end{lem}

\begin{proof}
Obviously, $M$ is well-defined and linear. And for any $u\in C_{0}(G),$

\[
Mu\sim\int_{0}^{1}\left(\int_{0}^{t}e^{s\Delta}\Delta u\text{ d\ensuremath{s}}\right)t^{-\frac{3}{2}}\text{d}t+\int_{1}^{\infty}\left(e^{t\Delta}u-u\right)t^{-\frac{3}{2}}\text{d}t.
\]
Then, 
\begin{align*}
\lVert Mu\lVert_{p} & \lesssim\int_{0}^{1}\left(\int_{0}^{t}\lVert e^{s\Delta}\lVert_{p\rightarrow p}\lVert\Delta u\lVert_{p}\text{ d\ensuremath{s}}\right)t^{-\frac{3}{2}}\text{d}t+\int_{1}^{\infty}\left(\lVert e^{t\Delta}\lVert_{p\rightarrow p}\lVert u\lVert_{p}+\lVert u\lVert_{p}\right)t^{-\frac{3}{2}}\text{d}t\\
 & \lesssim\lVert u\lVert_{p}\int_{0}^{1}t^{-\frac{1}{2}}\text{d}t+\lVert u\lVert_{p}\int_{1}^{\infty}t^{-\frac{3}{2}}\text{d}t\lesssim\lVert u\lVert_{p}.
\end{align*}
Hence, $M$ is a bounded operator in $\ell^{p}(G)$. And these definitions
are consistent for different $p$, i.e. two of them agree on their
common domain since the extensions of semigroup $e^{t\Delta}$ are
consistent in different $\ell^{p}(G)$, see \cite{KL}.

By the spectral mapping theorem in Banach spaces \cite[Proposition 3.1.1]{M},
we know that $L^{\frac{1}{2}}$ is a bounded operator in $\ell^{p}(G)$.
For $p=2$, using the functional calculus in Hilbert spaces, for any
$u\in\ell^{2}(G),$ we have 
\[
L^{\frac{1}{2}}u=Mu,
\]
which is also true for any $u\in C_{0}(G)$. Hence by the Bounded
Linear Transformation theorem \cite[Theorem I.7]{RS}, we get 
\[
L^{\frac{1}{2}}u=Mu,\:\forall u\in\ell^{p}(G).
\]
\end{proof}
Now we are ready to prove the discrete second-order Sobolev inequality. 
\begin{thm}
\label{thm: second order sobo ineq}For $N\geq3,1<p<\dfrac{N}{2},p^{\ast\ast}\coloneqq\dfrac{Np}{N-2p}$,
we have the second-order Sobolev inequalities 
\begin{equation}
\lVert u\rVert_{\ell^{p^{\ast\ast}}}\leq C_{p}\lVert u\rVert_{D^{2,p}},\;\forall u\in D_{0}^{2,p}(G),\label{eq: second-order sobo2}
\end{equation}
and for $1\leq p<\dfrac{N}{2}$, 
\begin{equation}
\lVert u\rVert_{\ell^{p^{\ast\ast}}}\leq C_{p}\lVert u\rVert_{\widetilde{D}^{2,p}},\;\forall u\in\widetilde{D}_{0}^{2,p}(G).\label{eq: second-order Sobolev inequ_hessian}
\end{equation}
\end{thm}

\begin{proof}
Using the completion trick, it suffices to prove that the second-order
Sobolev inequalities (\ref{eq: second-order sobo2}) and (\ref{eq: second-order Sobolev inequ_hessian})
hold for any $u\in C_{0}(G)$. By the first-order Sobolev inequality
(\ref{eq:first-order sobo ineq-1}) we have 
\begin{equation}
\lVert u\rVert_{\ell^{p^{\ast\ast}}}\leq C_{p}\lVert\nabla u\lVert_{p^{\ast}}.\label{eq:second ineq 1}
\end{equation}
By the boundedness of Riesz transforms \cite{N}, we know that 
\begin{equation}
\lVert\nabla u\lVert_{p^{\ast}}\lesssim\lVert L^{\frac{1}{2}}u\lVert_{p^{\ast}}.\label{eq:second ineq 2}
\end{equation}
And $L^{\frac{1}{2}}u\in\ell^{p}(G)$ since $L^{\frac{1}{2}}$ is
a bounded linear operator in $\ell^{p}(G)$ by Lemma \ref{lem:L^(1/2) is bounded}.
Then $L^{\frac{1}{2}}u\in D_{0}^{1,p}(G)$ follows from $\ell^{p}(G)$
embeds into $D^{1,p}(G)$ and $D_{0}^{1,p}(G)=D^{1,p}(G)$. Hence
by the first-order Sobolev inequality again we know 
\begin{equation}
\lVert L^{\frac{1}{2}}u\lVert_{p^{\ast}}\leq C_{p}\lVert\nabla L^{\frac{1}{2}}u\lVert_{p}.\label{eq:second ineq 3}
\end{equation}
By the boundedness of Riesz transforms, we get 
\begin{equation}
\lVert\nabla L^{\frac{1}{2}}u\lVert_{p}\lesssim\lVert L^{\frac{1}{2}}L^{\frac{1}{2}}u\lVert_{p}=\lVert u\rVert_{D^{2,p}}.\label{eq:second ineq 4}
\end{equation}
The inequality (\ref{eq: second-order sobo2}) is proved by (\ref{eq:second ineq 1})-(\ref{eq:second ineq 4}).

Next by the first-order Sobolev inequality (\ref{eq:first-order sobo ineq-1}),
\begin{align*}
\lVert u\rVert_{\ell^{p^{\ast\ast}}} & \lesssim\lVert\nabla u\rVert_{\ell^{p^{\ast}}}=\left(\stackrel[i=1]{m}{\sum}\lVert\partial_{i}u\rVert_{\ell^{p^{\ast}}}^{p^{\ast}}\right)^{1/p^{\ast}}\\
 & \lesssim\left(\stackrel[i=1]{m}{\sum}\lVert\nabla\partial_{i}u\rVert_{\ell^{p}}^{p^{\ast}}\right)^{1/p^{\ast}}\\
 & \lesssim\left(\left(\stackrel[i,j=1]{m}{\sum}\lVert\partial_{j}\partial_{i}u\rVert_{\ell^{p}}\right)^{p^{\ast}}\right)^{1/p^{\ast}}\sim\lVert\nabla^{2}u\rVert_{\ell^{p}}.
\end{align*}
Hence the Sobolev inequality (\ref{eq: second-order Sobolev inequ_hessian})
holds. 
\end{proof}
\begin{rem*}
(1) Since $\ell^{p}(G)$ embeds into $\ell^{q}(G)$ for any $q>p$,
see \cite[Lemma 2.1]{HLY}, we get the second-order Sobolev inequalities
(\ref{eq:second-order Sobo ineq}) and (\ref{eq: second-order Sobolev inequ_hessian})
in supercritical cases $q>p^{\ast\ast}$.

(2) The second-order Sobolev inequality for $u\in D_{0}^{2,1}(G)$
can not be proved directly since the Riesz transforms is weak type
$(1,1)$ for the nilpotent group $G$ \cite[Theorem 1]{N}, and we
don't know whether the inequality 
\[
\lVert u\rVert_{\ell^{1^{\ast}}}\lesssim\lVert\nabla u\rVert_{\ell^{1,\infty}}\coloneqq\underset{t>0}{\text{sup }}t\left(\sharp\left\{ x\in G:\mid\nabla u(x)\mid_{1}>t\right\} \right)
\]
holds or not. 
\end{rem*}
Then we prove the equivalence of higher-order Sobolev spaces. 
\begin{thm}
\label{thm: the equivalence for Z^N_k=00003D00003D00003D2}If $N\geq3$,
$1<p<\dfrac{N}{2}$, then 
\begin{equation}
D_{0}^{2,p}(\mathbb{Z}^{N})=D^{2,p}(\mathbb{Z}^{N}),\text{}\label{eq: D0kp(Z)=00003D00003D00003D00003D00003DDkp(Z)-1}
\end{equation}
and for $1\leq p<\dfrac{N}{2}$, 
\[
\widetilde{D}_{0}^{2,p}(\mathbb{Z}^{N})=\widetilde{D}^{2,p}(\mathbb{Z}^{N}).
\]
\end{thm}

\begin{proof}
By the same argument as Theorem \ref{thm:the equivalence for Z^N_k=00003D00003D00003D1},
$D_{0}^{2,p}(G)\subseteq D^{2,p}(G)$ and $\widetilde{D}_{0}^{2,p}(\mathbb{Z}^{N})\subseteq\widetilde{D}^{2,p}(\mathbb{Z}^{N})$
follow from the Sobolev inequalities (\ref{eq: second-order sobo2})
and (\ref{eq: second-order Sobolev inequ_hessian}).

In the other direction, the key is to find suitable cutoff functions
$\eta_{n}(x)\in C_{0}\left(\mathbb{Z}^{N}\right)$. For any $u\in D^{2,p}(\mathbb{Z}^{N})$(resp.
$\widetilde{D}^{2,p}(\mathbb{Z}^{N})$), set $u_{n}\coloneqq u\eta_{n}\in C_{0}\left(\mathbb{Z}^{N}\right)$,
then by H$\ddot{\text{o}}$lder inequalities 
\begin{equation}
\lVert u_{n}-u\rVert_{D^{2,p}(\mathbb{Z}^{N})}^{p}\lesssim\sum\limits _{x\in\mathbb{Z}^{N}}\lvert\Delta u(x)\mid^{p}\underset{y\sim x}{\text{max}}\mid\eta_{n}(y)-1\mid^{p}+\lVert\Delta\eta_{n}\rVert_{\ell^{\frac{N}{2}}}^{p}\lVert u\rVert_{\ell^{p^{\ast\ast}}}^{p}\label{eq:k=00003D00003D00003D2-1}
\end{equation}
and 
\begin{align}
\lVert u_{n}-u\rVert_{\tilde{D}^{2,p}(\mathbb{Z}^{N})}^{p} & \lesssim\sum\limits _{x\in\mathbb{Z}^{N}}\stackrel[i,j=1]{m}{\sum}\{\lvert\partial_{j}\partial_{i}u(x)\mid^{p}\mid\eta_{n}(xs_{i}s_{j})-1\mid^{p}+\lvert\partial_{j}u(x)\mid^{p}\mid\partial_{i}\eta_{n}(xs_{j})\mid^{p}\nonumber \\
 & +\lvert\partial_{i}u(x)\mid^{p}\mid\partial_{j}\eta_{n}(xs_{i})\mid^{p}+\lvert u(x)\mid^{p}\mid\partial_{j}\partial_{i}\eta_{n}(x)\mid^{p}\}\label{eq: k=00003D00003D00003D2-2}\\
 & \lesssim\sum\limits _{x\in\mathbb{Z}^{N}}\stackrel[i,j=1]{m}{\sum}\lvert\partial_{j}\partial_{i}u(x)\mid^{p}\underset{i,j}{\text{max}}\mid\eta_{n}(xs_{i}s_{j})-1\mid^{p}+\lVert\nabla\eta_{n}\rVert_{\ell^{N}}^{p}\lVert\nabla u\rVert_{\ell^{p^{\ast}}}^{p}\nonumber \\
 & +\lVert\nabla^{2}\eta_{n}\rVert_{\ell^{\frac{N}{2}}}^{p}\lVert u\rVert_{\ell^{p^{\ast\ast}}}^{p}.\nonumber 
\end{align}
Hence, if the cutoff functions satisfy 
\[
\eta_{n}\text{ is uniformly bounded,}
\]
\begin{equation}
\eta_{n}\to1\text{ pointwise on \ensuremath{\mathbb{Z}^{N}}}\label{eq: cut fun req 1-2}
\end{equation}
and

\[
\begin{cases}
\begin{array}{cc}
\lVert\Delta\eta_{n}\rVert_{\ell^{\frac{N}{2}}(\mathbb{Z}^{N})}\rightarrow0,\:\text{for }D^{2,p}(\mathbb{Z}^{N}),\\
\lVert\nabla\eta_{n}\rVert_{\ell^{N}(\mathbb{Z}^{N})},\:\lVert\nabla^{2}\eta_{n}\rVert_{\ell^{\frac{N}{2}}(\mathbb{Z}^{N})}\rightarrow0,\:\text{for }\widetilde{D}^{2,p}(\mathbb{Z}^{N}),
\end{array}\end{cases}
\]
then we can prove the other direction by the dominated convergence
theorem. We distinguish two cases and construct cutoff functions respectively:
\begin{figure}[h]
\includegraphics[scale=0.4]{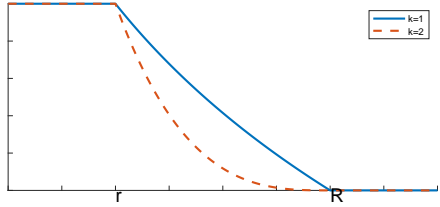}

\caption{The figure shows the cutoff functions constructed for $D^{1,p}(\mathbb{Z}^{N})$
(solid line) and $D^{2,p}(\mathbb{Z}^{N})$ (dotted line), respectively.}
\end{figure}

Case 1. For $D^{2,p}(\mathbb{Z}^{N})$, define $f(x)\coloneqq\sum\limits _{i}x_{i}^{2}$
and 
\[
\phi(s)\coloneqq1\wedge\frac{\textrm{log\ensuremath{(1-(1-\frac{s}{R})^{3})}}}{\textrm{log\ensuremath{(1-(1-\frac{r}{R})^{3})}}}\vee0.
\]
Let $\eta(x)\coloneqq\phi(f(x))$, $X\coloneqq\left\{ x\in\mathbb{Z}^{N}:f(y)>r\text{ for all \ensuremath{y\sim x}}\right\} $.
Then $\Delta f=2N$, $\Gamma f=4f+N$. The numbers $r$ and $R$ rather
stand for the square of the distance as $f$ is the square of the
Euclidean distance. By definition $\phi\in C^{2}\left(r,\infty\right)$
and 
\[
-(\text{log}\ensuremath{(\frac{R}{c_{1}r})})^{-1}\frac{c_{2}}{s}\leq\phi^{\prime}(s)\leq0\leq\phi^{\prime\prime}(s)\leq(\text{log}\ensuremath{(\frac{R}{c_{1}r})})^{-1}\frac{c_{3}}{s^{2}},\;\forall s\in(r,R].
\]
By the discrete chain rule (Lemma \ref{lem:chain rule}) on $X$,
we have 
\[
\Delta\eta(x)=\Delta\phi(f)(x)=2N\phi^{\prime}\left(f(x)\right)+(4f(x)+N)\phi^{\prime\prime}(\xi^{2})
\]
with $f(x)-2\sqrt{f(x)}+1\leq\xi\leq f(x)+2\sqrt{f(x)}+1$. And thus
if $100<r<f(y)$ for all $y\sim x$, we have $2\xi^{2}\geq f^{2}(x)$
giving 
\[
-(\text{log}\ensuremath{(\frac{R}{c_{1}r})})^{-1}2N\frac{c_{2}}{f(x)}\leq\Delta\eta(x)\leq(\text{log}\ensuremath{(\frac{R}{c_{1}r})})^{-1}(4f(x)+N)\frac{2c_{3}}{f^{2}(x)}.
\]
Then taking $c\coloneqq\text{max}\left\{ 2c_{2},9c_{3}\right\} $,
we get 
\[
\mid\Delta\eta(x)\mid\leq\frac{cN}{f(x)\text{log}\ensuremath{(\frac{R}{c_{1}r})}}\sim\frac{1}{f(x)\text{log}\ensuremath{\frac{R}{r}}}.
\]
Thus, 
\begin{align}
\sum\limits _{x\in X}\mid\Delta\eta(x)\mid^{\frac{N}{2}} & \lesssim\left(\frac{1}{\text{log}\ensuremath{\frac{R}{r}}}\right)^{\frac{N}{2}}\sum\limits _{r<f(x)\leq R}f(x)^{-\frac{N}{2}}\label{eq:estimate on X-1}\\
 & \lesssim\left(\frac{1}{\text{log}\ensuremath{\frac{R}{r}}}\right)^{\frac{N}{2}}\int\limits _{\widetilde{B}_{\sqrt{R}+1}\setminus\widetilde{B}_{\sqrt{r}-1}}\mid t\mid^{-N}\text{d\ensuremath{t}}\nonumber \\
 & \lesssim O\left(\textrm{\ensuremath{\left(\text{log}\ensuremath{\frac{R}{r}}\right)}}^{1-\frac{N}{2}}\right).\nonumber 
\end{align}
This goes to $0$ as $R\to\infty$ for $N\geq3$. The boundary at
$\sqrt{r}$ does not matter as we can keep $r$ constant and let $R\to\infty$.
At the neighbourhood of boundary $\partial\sqrt{r}$, i.e. $Y\coloneqq\left\{ x\in\mathbb{Z}^{N}:\sqrt{r}-2\leq\mid x\mid\leq\sqrt{r}+2\right\} $,
by the mean value theorem we have 
\[
\mid\Delta\eta(x)\mid=\mid\underset{y\sim x}{\sum}\left(\phi(f(y))-\phi(f(x))\right)\mid\lesssim\mid\phi'(r)\mid\mid f(y)-f(x)\mid\lesssim\frac{1}{\sqrt{r}\text{log}\ensuremath{\frac{R}{r}}},
\]
Since the cardinality of $Y$ is finite, then 
\begin{equation}
\sum\limits _{x\in Y}\mid\Delta\eta(x)\mid^{\frac{N}{2}}\lesssim\sum\limits _{x\in Y}\left(\frac{1}{\sqrt{r}\text{log}\ensuremath{\frac{R}{r}}}\right)^{\frac{N}{2}}\to0\text{ as }R\to\infty.\label{eq: estimate on Y-1}
\end{equation}
Hence, by (\ref{eq:estimate on X-1}) and (\ref{eq: estimate on Y-1})
we know 
\[
\lVert\Delta\eta(x)\rVert_{\ell^{\frac{N}{2}}(\mathbb{Z}^{N})}\to0\text{ as }R\to\infty,
\]
and we prove $u\in D_{0}^{2,p}(\mathbb{Z}^{N})$ by (\ref{eq:k=00003D00003D00003D2-1}).

Case 2. For $\widetilde{D}^{2,p}(\mathbb{Z}^{N})$, set $\eta(x)\coloneqq\phi(f(x))$
as in Case 1. Then we can check that 
\[
\lVert\nabla\eta\rVert_{\ell^{N}(\mathbb{Z}^{N})}^{N}\lesssim O\left(\textrm{\ensuremath{\left(\text{log}\ensuremath{\frac{R}{r}}\right)}}^{1-N}\right).
\]
By the norm equivalence (Lemma \ref{lem: norm equivalence}) and the
result in Case 1, we know 
\[
\lVert\nabla^{2}\eta\rVert_{\ell^{\frac{N}{2}}(\mathbb{Z}^{N})}^{\frac{N}{2}}\sim\lVert\Delta\eta\rVert_{\ell^{\frac{N}{2}}(\mathbb{Z}^{N})}^{\frac{N}{2}}\lesssim O\left(\textrm{\ensuremath{\left(\text{log}\ensuremath{\frac{R}{r}}\right)}}^{1-\frac{N}{2}}\right).
\]
Hence, $\widetilde{D}^{2,p}(\mathbb{Z}^{N})\subseteq\widetilde{D}_{0}^{2,p}(\mathbb{Z}^{N})$
by (\ref{eq: k=00003D00003D00003D2-2}). This finishes the proof. 
\end{proof}
\begin{rem*}
(1) The cutoff function $\eta$ defined for $D^{1,p}(\mathbb{Z}^{N})$
can not work for $D^{2,p}(\mathbb{Z}^{N})$ directly since the Laplacian
is \textquotedbl bad\textquotedbl{} when it comes to the sharp corner
$R$. We find a way to fix this by taking a \textquotedbl close to
linear\textquotedbl{} function $\phi$ being smooth at the boundary
$R$. Now the Euclidean distance is necessary for constructing the
cutoff function since its Laplacian is decreasing while the combinatorial
distance is not. Then using the chain rule we can get a good estimate.

(2) For $1<p<\dfrac{N}{2},$ by Lemma \ref{lem: norm equivalence}
and Theorem \ref{thm: the equivalence for Z^N_k=00003D00003D00003D2},
we know $D^{2,p}(\mathbb{Z}^{N})=D_{0}^{2,p}(\mathbb{Z}^{N})=\widetilde{D}_{0}^{2,p}(\mathbb{Z}^{N})=$$\widetilde{D}^{2,p}(\mathbb{Z}^{N})$. 
\end{rem*}
Then we prove the higher-order equivalence on a Cayley graph $G$
using the functional analysis. 
\begin{thm}
\label{thm: the equivalence for G_k=00003D00003D00003D2}If $N\geq3$,
$1<p<\dfrac{N}{2}$, then $D_{0}^{2,p}(G)=D^{2,p}(G)$. 
\end{thm}

\begin{proof}
First $D_{0}^{2,p}(G)\subseteq D^{2,p}(G)$ by the Sobolev inequalities
(\ref{eq: second-order sobo2}). The other direction can be proved
by checking the Laplace operator $\Delta$ is an isometry from $D_{0}^{2,p}(G)$
to $\ell^{p}(G)$. The injection follows from the $\ell^{p}-$Liouville
property of graphs \cite[Section 1.7]{B}. The range of $\Delta$,
denoted by $\text{Ran}(\Delta)$, is closed by the completeness of
$D_{0}^{2,p}(G)$. For any $h\in\ell^{q}(G)$, the dual space of $\ell^{p}(G)$,
\[
0=h(f),\:\forall f\in\text{Ran}(\Delta),
\]
which implies 
\[
0=\langle h,\Delta u\rangle=\langle\Delta h,u\rangle,\:\forall u\in C_{0}\left(G\right).
\]
Hence $\Delta h=0$ and $h=0$ by the $\ell^{p}-$Liouville property,
then the surjection follows from the closed range theorem. Similarly,
$\Delta$ is also an isometry from $D^{2,p}(G)$ to $\ell^{p}(G)$.
Then by the fact $D_{0}^{2,p}(G)\subseteq D^{2,p}(G)$ we know $D_{0}^{2,p}(G)=D^{2,p}(G)$. 
\end{proof}
\begin{rem*}
For the special case $p=2$, $D_{0}^{2,2}=D^{2,2}$ on general groups
satisfying the second-order Sobolev inequality, for example, polynomial
growth groups and non-amenable groups. 
\end{rem*}

\section{The Concentration-Compactness principle}

In this section, we prove the discrete Concentration-Compactness principle.
We first introduce a key lemma as in \cite[Theorem 1]{BL}.

Consider a measure space $(\Omega,\Sigma,\mu)$, which consists of
a set $\Omega$ equipped with a $\sigma$-algebra $\Sigma$ and a
Borel measure $\mu:\Sigma\to\left[0,\infty\right]$. 
\begin{lem}
\label{lem:(Discrete-Brezis-Lieb-lemma)}(Brézis-Lieb lemma) Let $(\Omega,\Sigma,\mu)$
be a measure space, $\{u_{n}\}\subset L^{p}(\Omega,\Sigma,\mu)$,
and $0<p<\infty.$ If \\
 (a)$\;\{u_{n}\}$ is uniformly bounded in $L^{p}$, and\\
 (b)$\;u_{n}\to u,n\to\infty$ $\mu$-almost everywhere in $\Omega$,
then 
\begin{equation}
\lim\limits _{n\to\infty}(\lVert u_{n}\rVert_{L^{p}}^{p}-\lVert u_{n}-u\rVert_{L^{p}}^{p})=\lVert u\rVert_{L^{p}}^{p}.\label{eq:weak}
\end{equation}
\end{lem}

\begin{rem*}
(1) The preceding lemma is a refinement of Fatou's Lemma.

(2) Since $\;\{u_{n}\}$ is uniformly bounded in $L^{p}$, passing
to a subsequence if necessary, we have 
\[
\lim\limits _{n\to\infty}\lVert u_{n}\rVert_{p}^{p}=\lim\limits _{n\to\infty}\lVert u_{n}-u\rVert_{p}^{p}+\lVert u\rVert_{p}^{p}.
\]

(3) If $\Omega$ is countable and $\mu$ is a positive measure defined
on $\Omega$, then we get a discrete version of Lemma \ref{lem:(Discrete-Brezis-Lieb-lemma)}. 
\end{rem*}
\begin{cor}
\label{cor:Corollary-1.} Let $\Omega\subset G$, ${\left\{ u_{n}\right\} }\subset D^{2,p}(G)$,
and $1<p<\infty$. If \\
 (a)$'$ $\left\{ {u_{n}}\right\} $ is uniformly bounded in $D^{2,p}(G)$,
and\\
 (b)$'$ $u_{n}\to u,n\to\infty$ pointwise on $G$, then 
\begin{equation}
\lim\limits _{n\to\infty}\left(\sum\limits _{x\in\Omega}\lvert\Delta u_{n}(x)\rvert^{p}-\sum\limits _{x\in\Omega}\lvert\Delta(u_{n}-u)(x)\rvert^{p}\right)=\sum\limits _{x\in\Omega}\lvert\Delta u(x)\rvert^{p}.\label{eq:gradient}
\end{equation}
\end{cor}

\begin{proof}
Since $\left\{ {\Delta u_{n}}\right\} $ is uniformly bounded in $\ell^{p}(G)$
and $\Delta u_{n}\rightarrow\Delta u,n\rightarrow\infty$ pointwise
on $G$, by Lemma \ref{lem:(Discrete-Brezis-Lieb-lemma)} we get 
\[
\lim\limits _{n\to\infty}(\lVert\Delta u_{n}\rVert_{\ell^{p}(\Omega)}^{p}-\lVert\Delta\left(u_{n}-u\right)\rVert_{\ell^{p}(\Omega)}^{p})=\lVert\Delta u\rVert_{\ell^{p}(\Omega)}^{p},
\]
which is equivalent to the equation (\ref{eq:gradient}). 
\end{proof}
Next, we establish the Concentration-Compactness principle on $G$. 
\begin{lem}
\label{lem:Concentration-Compact}(Discrete Concentration-Compactness
lemma) For $N\geq3,1<p<\frac{N}{2},q\geq p^{\ast\ast}$, if $\{u_{n}\}$
is uniformly bounded in $D^{2,p}(G)$. Then passing to a subsequence,
still denoted by $\{u_{n}\}$, we have 
\begin{equation}
u_{n}\longrightarrow u\quad\textrm{pointwise\;on}\;G,\label{eq:pointwise}
\end{equation}
\begin{equation}
\lvert\Delta u_{n}\rvert^{p}\xrightharpoonup{w^{\ast}}\lvert\Delta u\rvert^{p}\quad\textrm{in}\;\ell^{1}(G).\label{eq:w-star}
\end{equation}
And the following limits 
\[
\lim\limits _{R\to\infty}\lim\limits _{n\to\infty}\sum\limits _{d^{S}(x,e)>R}\lvert\Delta u_{n}\left(x\right)\rvert^{p}\coloneqq\mu_{\infty},\:\lim\limits _{R\to\infty}\lim\limits _{n\to\infty}\sum\limits _{d^{S}(x,e)>R}\lvert u_{n}(x)\rvert^{q}\coloneqq\nu_{\infty},
\]
exist. For the above $\left\{ u_{n}\right\} $, we have 
\begin{equation}
\lvert\Delta(u_{n}-u)\rvert^{p}\xrightharpoonup{w^{\ast}}0\quad\text{in}\;\ell^{1}(G),\label{eq:nu=00003D00003D00003D00003D00003D0}
\end{equation}
\begin{equation}
\lvert u_{n}-u\rvert^{q}\xrightharpoonup{w^{\ast}}0\quad\text{in}\;\ell^{1}(G),\label{eq:mu=00003D00003D00003D00003D00003D0}
\end{equation}
\begin{equation}
\nu_{\infty}^{p/q}\leq K^{-1}\mu_{\infty},\label{eq:infinite}
\end{equation}
\begin{equation}
\lim\limits _{n\to\infty}\lVert u_{n}\rVert_{D^{2,p}}^{p}=\lVert u\rVert_{D^{2,p}}^{p}+\mu_{\infty},\label{eq:composition1}
\end{equation}
\begin{equation}
\lim\limits _{n\to\infty}\lVert u_{n}\rVert_{q}^{q}=\lVert u\rVert_{q}^{q}+\nu_{\infty}.\label{eq:composition2}
\end{equation}
\end{lem}

\begin{proof}
Since $\{u_{n}\}$ is uniformly bounded in $\ell^{q}(G)$, they are
bounded in $\ell^{\infty}(G)$. By diagonal principle, passing to
a subsequence we get (\ref{eq:pointwise}). Since $\left\{ \lvert\Delta u_{n}\rvert^{p}\right\} $
is uniformly bounded in $\ell^{1}(G)$, we get (\ref{eq:w-star})
by the Banach-Alaoglu theorem and (\ref{eq:pointwise}). For every
$R\geq1$, passing to a subsequence if necessary, 
\[
\lim\limits _{n\to\infty}\sum\limits _{d^{S}(x,e)>R}\lvert\Delta u_{n}\left(x\right)\rvert^{p},\qquad\lim\limits _{n\to\infty}\sum\limits _{d^{S}(x,e)>R}\lvert u_{n}(x)\rvert^{q}.
\]
exist. Then we can define $\mu_{\infty}$, $\nu_{\infty}$ by the
monotonicity in $R$.

Let $v_{n}:=u_{n}-u$, then $v_{n}\to0$ pointwise on $G$ and $\left\{ \lvert\Delta v_{n}\rvert^{p}\right\} $
is uniformly bounded in $\ell^{1}(G)$. Then any subsequence of $\left\{ \lvert\Delta v_{n}\rvert^{p}\right\} $
contains a subsequence, still denoted by $\left\{ \lvert\Delta v_{n}\rvert^{p}\right\} $,
that $w^{\ast}$-converges to 0 in $\ell^{1}(G)$, which follows from
\[
\sum h\lvert\Delta v_{n}\rvert^{p}\to0,\text{ }\forall h\in C_{0}(G).
\]
Hence we obtain (\ref{eq:nu=00003D00003D00003D00003D00003D0}). Similarly,
we get (\ref{eq:mu=00003D00003D00003D00003D00003D0}).

For $R\geq1$, let $\Psi_{R}\in C(G)$ such that $\Psi_{R}(x)=1$
for $d^{S}(x,e)\geq R+1$, $\Psi_{R}(x)=0$ for $d^{S}(x,e)\leq R$.
By the discrete Sobolev inequality (\ref{eq:second-order Sobo ineq}),
we have 
\[
(\sum\lvert\Psi_{R}v_{n}\rvert^{q})^{p/q}\leq K^{-1}\sum\lvert\Delta(\Psi_{R}v_{n})\rvert^{p}=K^{-1}\underset{x}{\sum}\lvert\underset{y\sim x}{\sum}\left(\nabla_{xy}\Psi_{R}v_{n}(y)+\Psi_{R}(x)\nabla_{xy}v_{n}\right)\rvert^{p}.
\]
For any $\varepsilon>0$, there exists $C_{\varepsilon}>0$ such that
\[
\lvert\underset{y\sim x}{\sum}\nabla_{xy}\Psi_{R}v_{n}(y)+\underset{y\sim x}{\sum}\Psi_{R}(x)\nabla_{xy}v_{n}\rvert^{p}\leq C_{\varepsilon}\lvert\underset{y\sim x}{\sum}\nabla_{xy}\Psi_{R}v_{n}(y)\lvert^{p}+(1+\varepsilon)\lvert\Delta v_{n}\lvert^{p}\lvert\Psi_{R}(x)\lvert^{p}.
\]
Since $v_{n}\to0$ pointwise on $G$, by $\varepsilon\to0^{+}$ we
obtain 
\begin{equation}
\varlimsup\limits _{n\to\infty}(\sum\lvert\Psi_{R}v_{n}\rvert^{q})^{p/q}\leq K^{-1}\varlimsup\limits _{n\to\infty}\sum\lvert\Delta v_{n}\rvert^{p}\Psi_{R}^{p}.\label{eq:c-c1}
\end{equation}
From the definition of $\Psi_{R}$, 
\begin{equation}
\varlimsup\limits _{n\to\infty}(\sum\limits _{d^{S}(x,e)>R}\lvert v_{n}\rvert^{q})^{p/q}\leq K^{-1}\varlimsup\limits _{n\to\infty}\sum\limits _{d^{S}(x,e)>R}\lvert\Delta v_{n}\rvert^{p}.\label{eq: c-c2}
\end{equation}
By Lemma \ref{lem:(Discrete-Brezis-Lieb-lemma)} and Corollary \ref{cor:Corollary-1.},
we have 
\[
\lim\limits _{n\to\infty}\left(\sum\limits _{d^{S}(x,e)>R}\lvert\Delta u_{n}(x)\rvert^{p}-\sum\limits _{d^{S}(x,e)>R}\lvert\Delta v_{n}(x)\rvert^{p}\right)=\sum\limits _{d^{S}(x,e)>R}\lvert\Delta u(x)\rvert^{p},
\]
\[
\lim\limits _{n\to\infty}\left(\sum\limits _{d^{S}(x,e)>R}\lvert u_{n}(x)\rvert^{q}-\sum\limits _{d^{S}(x,e)>R}\lvert v_{n}(x)\rvert^{q}\right)=\sum\limits _{d^{S}(x,e)>R}\lvert u(x)\rvert^{q}.
\]
Hence by letting $R\to\infty$, we get 
\begin{equation}
\lim\limits _{R\to\infty}\lim\limits _{n\to\infty}\sum\limits _{d^{S}(x,e)>R}\lvert\Delta v_{n}(x)\rvert^{p}=\mu_{\infty},\label{eq:c-c4}
\end{equation}
\begin{equation}
\lim\limits _{R\to\infty}\lim\limits _{n\to\infty}\sum\limits _{d^{S}(x,e)>R}\lvert v_{n}(x)\rvert^{q}=\nu_{\infty}.\label{eq:c-c5}
\end{equation}
Combining the equations (\ref{eq: c-c2})-(\ref{eq:c-c5}), we get
\[
\nu_{\infty}^{p/q}\leq K^{-1}\mu_{\infty}.
\]

Since $u_{n}\to u$ pointwise on $G$, then for every $R\geq1$, 
\begin{align*}
\varlimsup\limits _{n\to\infty}\sum\lvert\Delta u_{n}\rvert^{p} & =\varlimsup\limits _{n\to\infty}\left(\sum\Psi_{R}\lvert\Delta u_{n}\rvert^{p}+\sum(1-\Psi_{R})\lvert\Delta u_{n}\rvert^{p}\right)\\
 & =\varlimsup\limits _{n\to\infty}\sum\Psi_{R}\lvert\Delta u_{n}\rvert^{p}+\sum(1-\Psi_{R})\lvert\Delta u\rvert^{p},
\end{align*}
and 
\begin{align*}
\varlimsup\limits _{n\to\infty}\sum\lvert u_{n}\rvert^{q} & =\varlimsup\limits _{n\to\infty}\left(\sum\Psi_{R}\lvert u_{n}\rvert^{q}+\sum(1-\Psi_{R})\lvert u_{n}\rvert^{q}\right)\\
 & =\varlimsup\limits _{n\to\infty}\sum\Psi_{R}\lvert u_{n}\rvert^{q}+\sum(1-\Psi_{R})\lvert u\rvert^{q}.
\end{align*}
Letting $R\to\infty$, we obtain 
\[
\lim\limits _{n\to\infty}\sum\lvert\Delta u_{n}\rvert^{p}=\mu_{\infty}+\sum\lvert\Delta u\rvert^{p}=\mu_{\infty}+\lVert u\rVert_{D^{2,p}}^{p},
\]
\[
\lim\limits _{n\to\infty}\sum\lvert u_{n}\rvert^{q}=\nu_{\infty}+\sum\lvert u\rvert^{q}=\nu_{\infty}+\lVert u\rVert_{q}^{q}.
\]
\end{proof}
\begin{rem*}
(1) In the continuous setting, P. L. Lions \cite{L3}, Bianchi et
al.\cite{BCS} and Ben-Naoum et al.\cite{BTW} proved that the limit
of the minimizing sequence norm can be divided into three parts, i.e.
the norm of the limit function, the norm of the limit of the difference
between the sequence and the limit function, and the norm of the sequence
at infinity. The corresponding parts still satisfy the Sobolev inequality,
see \cite[Lemma 1.40]{W2}.

(2) The difference between the sequence and the limit function $w^{\ast}$-converges
to $0$ in $\ell^{1}(G)$, i.e. (\ref{eq:nu=00003D00003D00003D00003D00003D0})
and (\ref{eq:mu=00003D00003D00003D00003D00003D0}), which is not true
in continuous setting. For example, consider the sequence of probability
measures $\left\{ \delta_{n}\right\} $ in $[0,1]$, where $\delta_{n}(x)\coloneqq n\chi_{[0,\frac{1}{n}]}\textrm{d}x$,
then $\delta_{n}\to0$ almost everywhere in $[0,1]$. However, $\delta_{n}\xrightharpoonup{w^{\ast}}\delta_{0}$
in $\left(C[0,1]\right)^{\ast}$ and the Dirac measure $\delta_{0}$
is non-zero. This is the advantage of the discrete setting. 
\end{rem*}

\section{Proof of Theorem \ref{thm:main1}}

In this section, we will prove the existence of the extremal function
for the discrete second-order Sobolev inequality (\ref{eq:second-order Sobo ineq}).
Firstly, we prove that the minimizing sequence after translation has
a uniform positive lower bound at the unit element $e\in G$. This
is crucial to rule out the vanishing case of the limit function. 
\begin{lem}
\label{lem:lower bound}For $N\geq3,1<p<\frac{N}{2},q>p^{\ast\ast}$,
let $\left\{ u_{n}\right\} \subset D^{2,p}(G)$ be a minimizing sequence
satisfying (\ref{eq:second-order Sobo ineq}). Then $\varliminf\limits _{n\to\infty}\lVert u_{n}\rVert_{\ell^{\infty}}>0$. 
\end{lem}

\begin{proof}
Choosing $q'$ such that $p^{\ast\ast}<q'<q<\infty$, by interpolation
inequality we have 
\[
1=\lVert u_{n}\rVert_{q}^{q}\leq\lVert u_{n}\rVert_{q'}^{q'}\lVert u_{n}\rVert_{\infty}^{q-q'}\leq C_{q',p}^{q'}\lVert u_{n}\rVert_{D^{2,p}}^{q'}\lVert u_{n}\rVert_{\infty}^{q-q'},
\]
where $C_{q',p}$ is the constant in the Sobolev inequality (\ref{eq:second-order Sobo ineq}).

By taking the limit, we obtain 
\[
1\leq C_{q',p}^{q'}K^{\frac{q'}{p}}\varliminf\limits _{n\to\infty}\lVert u_{n}\rVert_{\infty}^{q-q'}.
\]
This proves the lemma. 
\end{proof}
\begin{rem*}
The maximum of $\lvert u_{n}\rvert$ is attainable since $\lVert u_{n}\rVert_{q}=1$.
Define $v_{n}\left(x\right):=u_{n}(x_{n}x)$, where $\lvert u_{n}(x_{n})\rvert=\max\limits _{x}\lvert u_{n}(x)\rvert$.
Then the translation sequence $\left\{ v_{n}\right\} $ is uniformly
bounded in $D^{2,p}(G)$, $\lVert v_{n}\rVert_{q}=1$ and $\lvert v_{n}(e)\rvert=\lVert u_{n}\rVert_{\infty}$,
where $e$ is the unit element of $G$. By Lemma \ref{lem:lower bound},
passing to a subsequence if necessary, we have 
\[
v_{n}\longrightarrow v\quad\;\textrm{pointwise\;on}\;G,
\]
\begin{equation}
\lvert v(e)\rvert=\varliminf\limits _{n\to\infty}\lVert u_{n}\rVert_{\ell^{\infty}}>0.\label{eq:v(0)>0}
\end{equation}
\end{rem*}
Next, we give the proof %
\mbox{%
I%
} of Theorem \ref{thm:main1}. 
\begin{proof}[Proof %
\mbox{%
I%
} of Theorem \ref{thm:main1}]
Let $\left\{ u_{n}\right\} \subset D^{2,p}(G)$ be a minimizing sequence
satisfying (\ref{eq:min seq}). And the translation sequence $\left\{ v_{n}\right\} $
is defined in the Remark after Lemma \ref{lem:lower bound}.

By equalities (\ref{eq:composition2}) and (\ref{eq:composition1})
in Lemma \ref{lem:Concentration-Compact}, passing to a subsequence,
we get 
\[
K=\lim\limits _{n\to\infty}\lVert v_{n}\rVert_{D^{2,p}}^{p}=\lVert v\rVert_{D^{2,p}}^{p}+\mu_{\infty},
\]
\[
1=\lim\limits _{n\to\infty}\lVert v_{n}\rVert_{q}^{q}=\lVert v\rVert_{q}^{q}+\nu_{\infty}.
\]
From the Sobolev inequality (\ref{eq:second-order Sobo ineq}), (\ref{eq:infinite})
and the inequality 
\begin{equation}
\left(a^{q}+b^{q}\right)^{p/q}\leq a^{p}+b^{p}\text{,}\;\forall a,b\geq0,\label{eq:trivial inequality}
\end{equation}
we get 
\begin{align*}
K & =\lVert v\rVert_{D^{2,p}}^{p}+\mu_{\infty}\\
 & \geq K((\lVert v\rVert_{q}^{q})^{p/q}+\nu_{\infty}^{p/q})\\
 & \geq K(\lVert v\rVert_{q}^{q}+\nu_{\infty})^{p/q}=K.
\end{align*}
Since $\left(a^{q}+b^{q}\right)^{p/q}<a^{p}+b^{p}$ unless $a=0$
or $b=0$, we deduce from (\ref{eq:v(0)>0}) that $\lVert v\rVert_{q}^{q}=1$.
By 
\[
\lVert v\rVert_{D^{2,p}}^{p}\geq K\lVert v\rVert_{q}^{p},
\]
we get 
\[
\lVert v\rVert_{D^{2,p}}^{p}=K=\lim\limits _{n\to\infty}\lVert v_{n}\rVert_{D^{2,p}}^{p}.
\]
That is, $v$ is a minimizer. 
\end{proof}
Then we give another proof for Theorem \ref{thm:main1} using the
discrete Brézis-Lieb lemma. 
\begin{proof}[Proof %
\mbox{%
II%
} of Theorem \ref{thm:main1}]
Using Lemma \ref{lem:lower bound}, by the translation and taking
a subsequence if necessary as before, we can get a minimizing sequence
$\left\{ u_{n}\right\} $ satisfying (\ref{eq:min seq}), $u_{n}\to u$
pointwise on $G$, and $\lvert u(e)\rvert>0$.

By Lemma \ref{lem:(Discrete-Brezis-Lieb-lemma)}, the inequality (\ref{eq:trivial inequality})
and the Sobolev inequality, passing to a subsequence, we have 
\begin{align}
K=\underset{n\rightarrow\infty}{\lim}\lVert u_{n}\rVert_{D^{2,p}}^{p} & =\underset{n\rightarrow\infty}{\lim}\frac{\lVert u_{n}\rVert_{D^{2,p}}^{p}}{\lVert u_{n}\rVert_{q}^{p}}=\underset{n\rightarrow\infty}{\overline{\lim}}\frac{\lVert u_{n}-u\rVert_{D^{2,p}}^{p}+\lVert u\rVert_{D^{2,p}}^{p}}{\left(\lVert u_{n}-u\rVert_{q}^{q}+\lVert u\rVert_{q}^{q}\right)^{p/q}}\label{eq:p/q-1}\\
 & \geq\underset{n\rightarrow\infty}{\overline{\lim}}\frac{\lVert u_{n}-u\rVert_{D^{2,p}}^{p}+\lVert u\rVert_{D^{2,p}}^{p}}{\lVert u_{n}-u\rVert_{q}^{p}+\lVert u\rVert_{q}^{p}}\nonumber \\
 & \geq\underset{n\rightarrow\infty}{\overline{\lim}}\frac{K\lVert u_{n}-u\rVert_{q}^{p}+\lVert u\rVert_{D^{2,p}}^{p}}{\lVert u_{n}-u\rVert_{q}^{p}+\lVert u\rVert_{q}^{p}}.\nonumber 
\end{align}
Since $u\text{\ensuremath{\not\equiv}}0$, by the Sobolev inequality
we have that 
\[
\lVert u\rVert_{D^{2,p}}^{p}\leq K\lVert u\rVert_{q}^{p},
\]
which implies 
\[
\lVert u\rVert_{D^{2,p}}^{p}=K\lVert u\rVert_{q}^{p}.
\]
By (\ref{eq:p/q-1}), passing to a subsequence, we get 
\[
\underset{n\rightarrow\infty}{\lim}\lVert u_{n}-u\rVert_{D^{2,p}}^{p}=K\underset{n\rightarrow\infty}{\lim}\lVert u_{n}-u\rVert_{q}^{p}.
\]
Since $0<\lVert u\rVert_{q}\leq\underset{n\rightarrow\infty}{\lim}\lVert u_{n}\rVert_{q}=1$,
it suffices to show that $\lVert u\rVert_{q}=1$. Suppose that it
is not true, i.e. $0<\lVert u\rVert_{q}=D<1$, then by Lemma \ref{lem:(Discrete-Brezis-Lieb-lemma)},
\[
\underset{n\rightarrow\infty}{\lim}\lVert u_{n}-u\rVert_{q}^{q}=\underset{n\rightarrow\infty}{\lim}\lVert u_{n}\rVert_{q}^{q}-\lVert u\rVert_{q}^{q}=1-D^{q}>0.
\]
However, $\left(a^{q}+b^{q}\right)^{p/q}<a^{p}+b^{p}$ if $a,$ $b>0$.
This yields a contradiction by (\ref{eq:p/q-1}).

Thus, $\lVert u\rVert_{q}=1$ and $u$ is a minimizer. 
\end{proof}

\section{Proofs for corollaries }

As applications, we prove corollaries in the introduction. The following
lemma establishes the equivalence between solutions to the higher-order
quasilinear problem (\ref{eq:p-bihar}) and to the system (\ref{eq:LE system}). 
\begin{lem}
\label{lem: the equivalence}For $N\geq3$, $p,q>1$, $\frac{1}{p^{\prime}}+\frac{1}{q}<\frac{N-2}{N}$,
$p^{\prime}=\frac{p}{p-1}$, $q^{\prime}=\frac{q}{q-1}$, $(u,v)$
is a solution of (\ref{eq:p-bihar}) if and only if $u$ is a solution
of (\ref{eq:LE system}) with $v\coloneqq-\mid\Delta u\mid^{p-2}\Delta u$. 
\end{lem}

\begin{proof}
Since $\frac{1}{p^{\prime}}+\frac{1}{p^{\ast\ast}}=\frac{N-2}{N},$
$\frac{1}{p^{\prime}}+\frac{1}{q}<\frac{N-2}{N}$ if and only if $q>p^{\ast\ast}$.
Suppose that $u\in D^{2,p}(G)$ is a solution of (\ref{eq:p-bihar}).
Set $v\coloneqq-\mid\Delta u\mid^{p-2}\Delta u$. Then $v\in\ell^{p^{\prime}}(G)$
and $u\in\ell^{p^{\ast\ast}}(G)$ by the Sobolev inequality. And $u\in\ell^{q}(G)$,
since $\ell^{p^{\ast\ast}}(G)$ embeds into $\ell^{q}(G)$. Let $w$
be a solution of 
\[
-\Delta w=\lvert u\rvert^{q-2}u.
\]
Hence, $\Delta w\in\ell^{q^{\prime}}$ and $w\in D^{2,q^{\prime}}(G)$
by the fact that Laplace operator $\Delta$ is an isometry from $D^{2,q^{\prime}}(G)$
to $\ell^{q^{\prime}}(G)$, see Theorem \ref{thm: the equivalence for G_k=00003D00003D00003D2}.
By the uniqueness theorem for harmonic functions we know that 
\[
v=w\in D^{2,q^{\prime}}(G).
\]
Note that $\lvert v\rvert^{p^{\prime}-2}v=-\Delta u$. That is, $(u,v)$
is a solution of (\ref{eq:LE system}). On the other hand, if $(u,v)$
is a pair of solution for (\ref{eq:LE system}), we have $-\mid\Delta u\mid^{p-2}\Delta u=v$.
This proves the lemma. 
\end{proof}
By Theorem \ref{thm:main1}, we can prove Corollary \ref{cor:p-bihar}. 
\begin{proof}[Proof of Corollary \ref{cor:p-bihar}]
By Theorem \ref{thm:main1} there exists a minimizer $u$ for the
problem (\ref{eq:inf}). Let $v$ be a solution of 
\[
-\Delta v=\mid-\Delta u\mid.
\]
Then $\lVert v\rVert_{q}\lesssim\lVert\Delta u\rVert_{p}$ since Laplace
operator $\Delta$ is an isometry from $D_{0}^{2,p}(G)$ to $\ell^{p}(G)$,
and $D_{0}^{2,p}(G)$ embeds into $\ell^{q}(G)$ for $q\geq p^{\ast\ast}$.
Note that $u\leq v$ by the maximum principle. Replacing $u$ by $-u$,
we get $-u\leq v$ similarly. Hence, 
\[
0\leq\mid u\mid\leq v.
\]
In particular we have $1=\lVert u\rVert_{q}\leq\lVert v\rVert_{q}$
and $\lVert\Delta u\rVert_{p}=\lVert\Delta v\rVert_{p}$. Therefore,
we know that $\frac{v}{\lVert v\rVert_{q}}$ is a non-negative minimizer.
It follows from the Lagrange multiplier that $\frac{v}{\lVert v\rVert_{q}}$
is a non-negative solution of (\ref{eq:p-bihar}). The maximum principle
yields that it is positive. 
\end{proof}
Finally, we can prove Corollary \ref{cor:L-E system} by Corollary
\ref{cor:p-bihar} and Lemma \ref{lem: the equivalence}. 
\begin{proof}[Proof of Corollary \ref{cor:L-E system}]
By Corollary \ref{cor:p-bihar} there exists a positive solution
$u$ of equation (\ref{eq:p-bihar}). Set $v\coloneqq-\mid\Delta u\mid^{p-2}\Delta u$
as Lemma \ref{lem: the equivalence} and we know 
\[
-\Delta v=\lvert u\rvert^{q-2}u>0.
\]
Hence $v$ is positive by the maximum principle. 
\end{proof}

\section{The cases of $p=2$ and $p=1$}

In this section, we study the special cases of $p=2$ and $p=1$.
First, for $p=2$, consider a Cayley graph $(G,S)$ satisfying $\beta_{S}(n)\geq C(S)n^{N},\ \forall n\geq1$.
By the first-order Sobolev inequality (\ref{eq:first-order sobo ineq-1})
we know $D_{0}^{1,2}(G)\subseteq D^{1,2}(G)$. We fix an orientation
for edge set $E$. Then $D_{0}^{1,2}(G)$ and $D^{1,2}(G)$ are Hilbert
spaces equipped with the inner product 
\[
\langle\nabla_{e}u,\nabla_{e}v\rangle_{E}\coloneqq\underset{e\in E}{\sum}\left(u(e_{+})-u(e_{-})\right)\left(v(e_{+})-v(e_{-})\right),
\]
where $e_{-}$ and $e_{+}$ are the initial and terminal endpoints
of $e$ respectively. For any $\alpha\in C(E)$, 
\[
\text{div}\,\alpha(x)\coloneqq\underset{e=(y,x)\in E}{\sum}\alpha(e)-\underset{\tilde{e}=(x,y)\in E}{\sum}\alpha(\tilde{e}),\;\;x\in V.
\]
Then we can prove the Hodge decomposition theorem on $1$-forms on
edges on the Cayley graph. 
\begin{thm}
\label{thm:Hodge decomposition}If the Cayley graph $(G,S)$ satisfies
$\beta_{S}(n)\geq C(S)n^{N},\ \forall n\geq1$, then we have decompositions
\[
\ell^{2}(E)=\nabla D_{0}^{1,2}(G)\oplus H
\]
and 
\[
\ell^{2}(E)=\nabla D^{1,2}(G)\oplus H,
\]
where $H\coloneqq\left\{ u\in\ell^{2}(E):\text{div}\,u=0\right\} $. 
\end{thm}

\begin{proof}
For any $u\in\ell^{2}(E)$, we define a bounded linear operator on
$D_{0}^{1,2}(G)$ as 
\[
L(v)\coloneqq-\langle\text{div}\,u,v\rangle.
\]
A nondegerate bilinear functional $B:D_{0}^{1,2}(G)\times D_{0}^{1,2}(G)\to\mathbb{R}$
is defined as 
\[
B(v,w)\coloneqq\langle\nabla_{e}v,\nabla_{e}w\rangle_{E}.
\]
By the Lax--Milgram theorem, there exists $f\in D_{0}^{1,2}(G)$
such that 
\[
B(f,v)=L(v),\:\forall v\in D_{0}^{1,2}(G).
\]
Hence, for any $v\in C_{0}(G)$, 
\[
\langle\Delta f,v\rangle=-\langle\nabla_{e}f,\nabla_{e}v\rangle_{E}=\langle\text{div}\,u,v\rangle,
\]
which implies that $\Delta f=\text{div}\,u$.

Since $\lVert\nabla f\rVert_{\ell^{2}(E)}=\frac{1}{2}\lVert f\rVert_{D^{1,2}(G)}$,
$\nabla f\in\ell^{2}(E)$. Let $h=u-\nabla f\in\ell^{2}(E)$. Then
\[
\text{div}\,h=\text{div}\,u-\text{div}\,\nabla f=\text{div}\,u-\Delta f=0.
\]
That is $h\in H$. The decomposition is proved. If $u\in\nabla D_{0}^{1,2}(G)\cap H$,
then there exists $v\in D_{0}^{1,2}(G)$ such that $u=\nabla v$ and
$\text{div}\,u=\Delta v=0$. Hence $v=0$ and $u=0$. The property
of the direct sum is proved. And we can get the decomposition $\ell^{2}(E)=\nabla D^{1,2}(G)\oplus H$
by the same argument. 
\end{proof}
\begin{cor}
\label{cor:D012=00003D00003D00003D00003DD12}If the Cayley graph $(G,S)$
satisfies $\beta_{S}(n)\geq C(S)n^{N},\ \forall n\geq1$, then $D_{0}^{1,2}(G)=D^{1,2}(G)$. 
\end{cor}

\begin{proof}
By the Sobolev inequality we know $D_{0}^{1,2}(G)\subseteq D^{1,2}(G)$.
Then for any $u\in D^{1,2}(G)$, by the Hodge decomposition theorem
we get 
\[
\nabla u=\nabla\tilde{u}+h,\,\tilde{u}\in D_{0}^{1,2}(G),h\in H.
\]
And $h=\nabla u-\nabla\tilde{u}\in\nabla D^{1,2}(G)$, hence $h=0$,
that is $\nabla\left(u-\tilde{u}\right)=0$, which implies that $u=\tilde{u}\in D_{0}^{1,2}(G)$.
Hence $D^{1,2}(G)\subseteq D_{0}^{1,2}(G)$. 
\end{proof}
\begin{rem*}
For the groups satisfying the second-order Sobolev inequality, for
example, polynomial growth groups and non-amenable groups, we can
prove $D_{0}^{2,2}=D^{2,2}$ by checking the Laplace operator is an
isometry from $D_{0}^{2,2}(G)$ (also $D^{2,2}(G)$) to $\ell^{2}(G)$
as Theorem \ref{thm: the equivalence for G_k=00003D00003D00003D2}. 
\end{rem*}
For the special case $p=1$, since $\widetilde{D}_{0}^{2,1}(\mathbb{Z}^{N})=\widetilde{D}^{2,1}(\mathbb{Z}^{N})$
by Lemma \ref{thm: the equivalence for Z^N_k=00003D00003D00003D2},
we can consider the optimal constant $\widetilde{K}$ of Sobolev inequality
(\ref{eq: second-order Sobolev inequ_hessian}) in the supercritical
case $q>1^{\ast\ast}=\frac{N}{N-2}$ on $\mathbb{Z}^{N}$: 
\begin{equation}
\widetilde{K}:=\inf_{\substack{u\in\widetilde{D}^{2,1}(\mathbb{Z}^{N})\\
\lVert u\rVert_{q}=1
}
}\lVert u\rVert_{\widetilde{D}^{2,1}}.\label{eq:inf-1}
\end{equation}
In order to prove that the infimum is achieved, consider a minimizing
sequence $\{u_{n}\}\subset\widetilde{D}^{2,1}(\mathbb{Z}^{N})$ satisfying
\begin{equation}
\lVert u_{n}\rVert_{q}=1,\lVert u_{n}\rVert_{\widetilde{D}^{2,1}(\mathbb{Z}^{N})}\to\widetilde{K},n\to\infty.\label{eq:min seq-1}
\end{equation}
Then by the same argument as the proof of Theorem \ref{thm:main1}
we can get the following result for $\widetilde{D}^{2,1}(\mathbb{Z}^{N})$. 
\begin{thm}
\label{thm:main1-1}For $N\geq3,$ $q>1^{\ast\ast}=\frac{N}{N-2}$,
let $\left\{ u_{n}\right\} \subset\widetilde{D}^{2,1}(\mathbb{Z}^{N})$
be a minimizing sequence satisfying (\ref{eq:min seq-1}). Then there
exists a sequence $\{x_{n}\}\subset\mathbb{Z}^{N}$ and $v\in\widetilde{D}^{2,1}$
such that the sequence after translation $\left\{ v_{n}(x):=u_{n}(x_{n}+x)\right\} $
contains a convergent subsequence that converges to $v$ in $\widetilde{D}^{2,1}$.
And $v$ is a minimizer for $\widetilde{K}$. 
\end{thm}

\begin{rem*}
This result implies that the best constant can be obtained in the
supercritical case. 
\end{rem*}
$\mathbf{\boldsymbol{\mathbf{Acknowledgements}\mathbf{}\text{:}}}$
The authors would like to thank Xueping Huang, Genggeng Huang and
Fengwen Han for helpful discussions and suggestions. B.H. is supported
by NSFC, no.11831004.

 \bibliographystyle{plain}
\bibliography{../../Documents/biharmonic_ref}

\end{document}